\documentclass[11pt]{article}

\usepackage{etex}

\usepackage{mathtools}

\usepackage{enumitem}

\usepackage{amsmath,amsthm}

\usepackage{csquotes}

\usepackage{todonotes}

\makeatletter 
\let\orgdescriptionlabel\descriptionlabel
\renewcommand*{\descriptionlabel}[1]{%
	\let\orglabel\label
	\let\label\@gobble
	\phantomsection
	\edef\@currentlabel{#1}%
		\let\label\orglabel
	\orgdescriptionlabel{#1}%
}
													
\def\th@plain{%
	\thm@notefont{} % same as heading font
	\itshape % body font
}
\def\th@definition{%
	\thm@notefont{}% same as heading font
	\normalfont % body font
}
\g@addto@macro\th@remark{\thm@headpunct{}}
\g@addto@macro\th@definition{\thm@headpunct{}}
\g@addto@macro\th@plain{\thm@headpunct{}}
\makeatother % change the ''category code" of @ back to 12

\usepackage{amssymb}
\usepackage{graphicx}

\usepackage[final]{showkeys}

\usepackage{etoolbox}

\usepackage{wasysym}

\usepackage{mathrsfs}

\usepackage{mathpazo}

\usepackage[titletoc]{appendix}
\usepackage[doc]{optional}
\usepackage{soul}

\usepackage{colortbl,booktabs,sectsty,multirow}
\usepackage{xcolor}

\usepackage{cancel}

\usepackage{empheq}

\definecolor{myblue}{rgb}{.8, .8, 1}
  \newcommand*\mybluebox[1]{%
    \colorbox{myblue}{\hspace{1em}#1\hspace{1em}}}

\usepackage[obeyspaces,hyphens,spaces]{url}

\usepackage{hyperref}
\hypersetup{
    colorlinks=true,       % false: boxed links; true: colored links
    linkcolor=blue,          % color of internal links (change box color with linkbordercolor)
    citecolor=blue,        % color of links to bibliography
    filecolor=magenta,      % color of file links
    urlcolor=cyan           % color of external links
}

\usepackage[
  open,
  openlevel=2,
  atend,
  numbered
]{bookmark}

\usepackage[capitalize, nameinlink, noabbrev]{cleveref}
\crefname{equation}{}{}
\crefname{chapter}{Chapter}{Chapter}
\crefname{item}{}{items}
\crefname{figure}{Figure}{Figure}
\crefname{theorem}{Theorem}{Theorem}
\crefname{lemma}{Lemma}{}
\crefname{proposition}{Proposition}{Proposition}
\crefname{corollary}{Corollary}{Corollary}
\crefname{definition}{Definition}{Definition}
\crefname{fact}{Fact}{Fact}
\crefname{example}{Example}{Example}
\crefname{algorithm}{Algorithm}{Algorithm}
\crefname{remark}{Remark}{Remark}
\crefname{note}{Note}{Note}
\crefname{notation}{Notation}{Notation}
\crefname{case}{Case}{Case}
\crefname{exercise}{Exercise}{Exercise}
\crefname{question}{Question}{Question}
\crefname{claim}{Claim}{Claim}
\crefname{enumi}{}{}

\usepackage[top= 2cm, bottom = 2 cm, left = 2.2 cm, right= 2.2 cm]{geometry}

\usepackage{float}

\usepackage{pgf}

\usepackage{array}
\usepackage{tabu}
\allowdisplaybreaks
\numberwithin{equation}{section}

\theoremstyle{plain} 
\newtheorem{theorem}{Theorem}[section]

\newtheorem{corollary}[theorem]{Corollary}
\newtheorem{fact}[theorem]{Fact}
\newtheorem{lemma}[theorem]{Lemma}
\newtheorem{proposition}[theorem]{Proposition}

\theoremstyle{definition} %Set the text in roman and add extral space above and below
\newtheorem{definition}[theorem]{Definition}
\newtheorem{example}[theorem]{Example}

\newtheorem{question}[theorem]{Question}
\newtheorem{remark}[theorem]{Remark}

\newcommand{\aff}{\ensuremath{\operatorname{aff}}}

\newcommand{\spn}{\ensuremath{{\operatorname{span}}}}
\newcommand{\weakly}{\ensuremath{{\;\operatorname{\rightharpoonup}\;}}}
\newcommand{\HH}{\ensuremath{{\mathcal{H}}}}
\newcommand{\dom}{\ensuremath{\operatorname{dom}}}

\newcommand{\Fix}{\ensuremath{\operatorname{Fix}}}
\newcommand{\Id}{\ensuremath{\operatorname{Id}}}

\newcommand{\Pro}{\ensuremath{\operatorname{P}}}
\newcommand{\R}{\ensuremath{\operatorname{R}}}
\newcommand{\I}{\ensuremath{\operatorname{I}}}
\newcommand{\card}{\ensuremath{\operatorname{card}}}

\newcommand{\CCO}[1]{CC{#1}}

\newcommand{\CC}[1]{CC_{#1}}

\providecommand{\norm}[1]{\lVert#1\rVert}
\providecommand{\Norm}[1]{{\Big\lVert}#1{\Big\rVert}}

\providecommand{\innp}[1]{\langle#1\rangle}

\newcommand\scalemath[2]{\scalebox{#1}{\mbox{\ensuremath{\displaystyle #2}}}}

\begin{document}
%set title in \sffamily font
\title{ \sffamily  On circumcenter mappings induced by nonexpansive operators}

\author{
         Heinz H.\ Bauschke\thanks{
                 Mathematics, University of British Columbia, Kelowna, B.C.\ V1V~1V7, Canada.
                 E-mail: \href{mailto:heinz.bauschke@ubc.ca}{\texttt{heinz.bauschke@ubc.ca}}.},~
         Hui\ Ouyang\thanks{
                 Mathematics, University of British Columbia, Kelowna, B.C.\ V1V~1V7, Canada.
                 E-mail: \href{mailto:hui.ouyang@alumni.ubc.ca}{\texttt{hui.ouyang@alumni.ubc.ca}}.},~ 
         and Xianfu\ Wang\thanks{
                 Mathematics, University of British Columbia, Kelowna, B.C.\ V1V~1V7, Canada.
                 E-mail: \href{mailto:shawn.wang@ubc.ca}{\texttt{shawn.wang@ubc.ca}}.}
                 }

\date{November 27, 2018} 

\maketitle

%---------------------------------------------Abstract---------------------------------------------
\begin{abstract}
\noindent
We introduce the circumcenter mapping induced by a set of (usually nonexpansive) operators. 
One prominent example of a circumcenter mapping is the celebrated
Douglas--Rachford splitting operator.
Our study is motivated by the Circumcentered--Douglas--Rachford
method recently introduced by Behling, Bello Cruz, and Santos in order to accelerate 
the Douglas--Rachford method for solving certain classes of feasibility problems.
We systematically explore the properness of the circumcenter mapping induced
by reflectors or projectors. Numerous examples are presented. 
We also present a version of Browder's demiclosedness principle for circumcenter mappings. 
\end{abstract}

{\small
\noindent
{\bfseries 2010 Mathematics Subject Classification:} 
{Primary 47H09,47H04; Secondary 
41A50, 
90C25 
}

\noindent{\bfseries Keywords:}
Browder's demiclosedness principle, 
circumcenter, 
circumcenter mapping,
nonexpansive,
projector,
reflector.
}

%---------------------------------------------Introduction---------------------------------------------
\section{Introduction}
Throughout this paper, we assume that 
	\begin{empheq}[box = \mybluebox]{equation*}
		 \text{$\mathcal{H}$ is a real Hilbert space} 
	\end{empheq}
with inner product $\innp{\cdot,\cdot}$ and induced norm $\|\cdot\|$. 
Let $m \in \mathbb{N} \smallsetminus \{0\}$, and let
$T_{1}, \ldots, T_{m-1}, T_{m}$ be operators 
from $\mathcal{H}$ to $\mathcal{H}$. Set 
\begin{empheq}{equation*}
\mathcal{S}=\{ T_{1}, \ldots, T_{m-1}, T_{m} \}, 
\end{empheq}
and denote the power set of $\mathcal{H}$ as $2^{\mathcal{H}}$. The associated set-valued operator $\mathcal{S}: \mathcal{H} \rightarrow 2^{\mathcal{H}}$ is defined by
\begin{empheq}{equation*}
(\forall x \in \mathcal{H}) \quad \mathcal{S}(x)=\{ T_{1}x, \ldots, T_{m-1}x, T_{m}x\}.
\end{empheq}

Unless otherwise specified, we assume that
\begin{empheq}[box = \mybluebox]{equation*}
U_{1}, \ldots, U_{m}~\text{are closed affine subspaces of}~\mathcal{H},
\text{ with }\bigcap^{m}_{i=1} U_{i} \neq \varnothing. 
\end{empheq}

In this paper, we introduce the circumcenter mapping $\CC{\mathcal{S}}$
induced by $\mathcal{S}$ which maps every element $x \in \mathcal{H}$ to
either empty set or the (unique if it exists) circumcenter of the finitely
many elements in the nonempty set $\mathcal{S}(x)$. 
In fact, the circumcenter mapping $\CC{\mathcal{S}}$ induced by $\mathcal{S}$
is the composition $\CCO{} \circ \mathcal{S}$ where $\CCO{}$ is the
circumcenter operator defined in \cite{BOyW2018}. 
The domain $\CC{\mathcal{S}}$ is defined to be $\dom
\CC{\mathcal{S}} = \{ x \in \mathcal{H} ~|~ \CC{\mathcal{S}}x \neq
\varnothing\}$.
We say the circumcenter
mapping $\CC{\mathcal{S}}$ is \emph{proper}, if $ \dom \CC{\mathcal{S}}(x) =
\mathcal{H}$. 
Properness is an important property for algorithms where one wishes to
consider sequences of the form $(\CC{\mathcal{S}}^{k}x)_{k \in \mathbb{N}}$. 

\emph{The goal of this paper is to explore conditions sufficient for the
circumcenter mapping to be proper. We also connect the circumcenter mapping
to the celebrated demiclosedness principle by Felix Browder.}

The CRM (Circumcentered--Reflection Method) operator $C$ recently 
investigated by Behling,
Bello Cruz, and Santos in {\rm \cite[page~159]{BCS2018}} is a particular instance of
a proper circumcenter mapping. The C--DRM
(Circumcentered--Douglas--Rachford Method) operator $C_{T}$ defined by 
Behling et al.\ in {\rm \cite[Section~2]{BCS2017}} is CRM operator 
associated with only two linear subspaces. Hence, the $C_{T}$ is a 
special case of our proper circumcenter mapping as well.

Behling et al.\ introduced in  \cite{BCS2017} the C--DRM which generates 
iterates by taking the intersection of bisectors of reflection steps to
accelerate the Douglas--Rachford method to solve certain classes of
feasibility problems. Our paper \cite{BOyW2018} and this paper are
motivated by \cite{BCS2017}. The proof of one of our main results, 
\cref{thm:CCS:proper}, is inspired by
{\rm\cite[Lemma~2]{BCS2017}}.
We now discuss further results and the organization of this paper.
In \cref{sec:Preliminaries}, we collect various results for subsequent use.
In particular, facts on circumcenter operator defined in {\rm
\cite[Definition~3.4]{BOyW2018}} are reviewed in Section~\ref{Sec:Subsec:CircOpe}. 
In \cref{sec:CircumMapping}, we introduce the circumcenter mapping
$\CC{\mathcal{S}}$ induced by a set of operators $\mathcal{S}$. Based on some
known results of circumcenter operator, we derive some sufficient conditions
for the circumcenter mapping $\CC{\mathcal{S}}$ to be proper. When
$\mathcal{S}$ consists of only three operators, we provide a sufficient and
necessary condition for the $\CC{\mathcal{S}}$ to be proper. 
We also obtain conditions sufficient for continuity. 
Examples illustrating the tightness of our assumptions are provided as well. 
\cref{subsec:Demiclosedness} contains the demiclosedness principle
for certain circumcenter mappings. 
In \cref{sec:CircumMappingReflectors}, we consider the circumcenter of 
finite subsets drawn from the affine hull of compositions of reflectors.
Inspired by {\rm \cite[Lemma~2]{BCS2017}}, we prove the properness of 
a certain 
class of circumcenter mappings induced by reflectors. 
We also provide improper examples. 
Two particular instances of $\CC{\mathcal{S}}$, one
of which belongs to the class of C--DRM operators from \cite{BCS2017} while the
other is new, are considered. Comparing
to the Douglas--Rachford Method (DRM) and the Method of Alternating projections
(MAP), we find in preliminary numerical explorations that 
$(\CC{\mathcal{S}}^{k}x)_{k \in \mathbb{N}}$ can be used to solve best approximation
problems. 
It is interesting that in general $\CC{\mathcal{S}}$ is neither
continuous nor linear.
In \cref{sec:CircumMappingProjectots}, the operators in $\mathcal{S}$ are
chosen from the affine hull of the set of compositions of projectors. We provide
both proper and improper examples of corresponding circumcenter mappings. 
The final \cref{sec:MoreImproper} deals with reflectors and reflected resolvents. 

Let us turn to notation. Let $K$ and $C$ be subsets of $\mathcal{H}$, 
$z \in \mathcal{H}$ and $\lambda \in \mathbb{R}$. Then $K+C =\{x +y
~|~ x \in K, y \in C \}$, $K +z = K + \{z\}$, and $\lambda K =\{\lambda x ~|~
x \in K\}$. The cardinality of the set $K$ is denoted as $\card(K)$. The
intersection of all the linear subspaces of $\mathcal{H}$ containing $K$ is
called the \emph{span} of $K$, and is denoted by $\spn~K$. A nonempty subset
$K$ of $\mathcal{H}$ is an \emph{affine subspace} of $\mathcal{H}$ if
$(\forall \rho\in\mathbb{R})$ $\rho K + (1-\rho)K = K$; moreover, the
smallest affine subspace containing $K$ is the
\emph{affine hull} of $K$, denoted $\aff K$. Assume that $C$ is a nonempty
closed, convex subset in $\mathcal{H}$. We denote by $\Pro_{C}$ the
\emph{projector} onto $C$. $\R_{C} := 2\Pro_{C} -\Id $ is the
\emph{reflector} associated with $C$. Let $T:\mathcal{H} \rightarrow
\mathcal{H}$. The set of \emph{fixed points} of $T$ is $\Fix
T =\{x \in \mathcal{H} ~|~ x= Tx\}$.
Let $(x_{k})_{k \in \mathbb{N}}$ be a sequence in $\mathcal{H}$ and let $x \in \mathcal{H}$. 
We use $x_{k} \rightharpoonup x$ to indicate that $(x_{k})_{k \in \mathbb{N}}$
converges weakly to $x$. 
The set $\mathbf{B}[x;r] := \{y \in \mathcal{H} ~|~ \norm{y-x} \leq r\}$ 
is the closed ball centered at $x$ of radius $r\geq 0$.
For other notation not explicitly defined here, we refer the reader to \cite{BC2017}.

%---------------------------------------------Preliminaries---------------------------------------------
\section{Auxiliary results} \label{sec:Preliminaries}
In this section, we provide various results that will be useful in the sequel. We start with some facts about affine subspaces.

\subsection{Affine subspaces and related concepts}

\begin{definition} {\rm \cite[page~4]{R1970}}
An affine subspace $C$ is said to be \emph{parallel} to an affine subspace $
M $ if $C = M +a $ for some $ a \in \mathcal{H}$.
\end{definition}

\begin{fact} {\rm \cite[Theorem~1.2]{R1970}} 
\label{fac:AffinePointLinearSpace}
	Every affine subspace $C$ is parallel to a unique linear
	subspace $L$, which is given by 
	\begin{align*}
	(\forall y \in C) \quad L = C - y = C - C . 
	\end{align*}
\end{fact}

\begin{definition} {\rm \cite[page~4]{R1970}}
 The \emph{dimension} of a nonempty affine subspace is defined to be the dimension of the 
 linear subspace parallel to it. 
\end{definition}

\begin{fact} {\rm \cite[page~7]{R1970}} \label{fac:AffSubsExpre}
  Let $x_{1}, \ldots, x_{m} \in \mathcal{H}$. 
  Then the affine hull is given by 
	\begin{align*}
	\aff\{x_{1}, \ldots,  x_{m}\}
	=\Big\{\lambda_{1}x_{1}+\cdots +\lambda_{m}x_{m} ~\Big|~\lambda_{1},\ldots,\lambda_{m} \in \mathbb{R} ~\text{and}~\sum^{m}_{i=1} \lambda_{i}=1 \Big\}.
	\end{align*}
\end{fact}

\begin{fact}  {\rm \cite[Lemma~2.6]{BOyW2018}} \label{fact:AffineHull}
	Let $x_{1}, \ldots,x_{m} \in \mathcal{H}$, where $m \geq 2$. Then for every
	$i_{0} \in \{2, \ldots, m\}$, we have
	\begin{align*}
	\aff\{x_{1}, \ldots, x_{m}\} 
	&~=x_{1} + \spn \{x_{2}-x_{1}, \ldots, x_{m}-x_{1}\}\\
	&~=x_{i_{0}}+\spn\{x_{1}-x_{i_{0}},\ldots,x_{i_{0}-1}-x_{i_{0}}, x_{i_{0}+1}-x_{i_{0}},\ldots,x_{m}-x_{i_{0}}\}.
	\end{align*} 
\end{fact}

\begin{definition} {\rm \cite[page~6]{R1970}}
Let $x_{0}, x_{1}, \ldots, x_{m} \in \mathcal{H}$. The $m+1$ vectors $x_{0}, x_{1}, \ldots, x_{m}$ are said to be affinely independent if $\aff \{x_{0}, x_{1}, \ldots, x_{m}\}$ is $m$-dimensional. We will also say $(x_{0},x_{1},\ldots, x_{m}) = (x_{i})_{i \in \{0,1,\ldots,m\}}$ is affinely independent.
\end{definition}

\begin{fact}{\rm \cite[page~7]{R1970}} \label{fac:AffinIndeLineInd}
Let $x_{1}, x_{2}, \ldots, x_{m} \in \mathcal{H}$. Then $x_{1}, x_{2}, \ldots,x_{m}$ are affinely independent if and only if $ x_{2}-x_{1}, \ldots, x_{m}-x_{1}$ are linearly independent.
\end{fact}

\subsection{Projectors and reflectors}
Our first result follows easily from the definitions. 

\begin{lemma} \label{lem:PU:RUIdempotent}    
Let $C$ be a nonempty closed convex subset of $\mathcal{H}$. 
Then 
\begin{enumerate}
\item \label{lem:PUIdempotent} $ \Pro_{C} \Pro_{C}=\Pro_{C}$.
\item \label{lem:RUIdempotent} $\R_{C}\R_{C} = \Id$.
\end{enumerate}
\end{lemma}

\begin{fact} {\rm \cite[Theorem~5.8]{D2012}} \label{MetrProSubs8}
Let $C$ be a closed linear subspace of $\mathcal{H}$. Then
\begin{enumerate}
\item \label{MetrProSubs8:ii} $\Id =\Pro_{C}+\Pro_{C^{\perp}}$.
\item \label{MetrProSubs8:iv} $C^{\perp}=\{x \in \mathcal{H}~|~ \Pro_{C}(x)=0\}$ and $C=\{x \in \mathcal{H}~|~ \Pro_{C^{\perp}}(x)=0\}=\{x \in \mathcal{H}~|~ \Pro_{C}(x)=x \}$.
\end{enumerate}
\end{fact}

The following result is a mild extension {\rm \cite[Proposition~1]{BCS2017}}
and it is useful in the proof of \cref{thm:CCS:proper}.

\begin{proposition}  \label{prop:PR}
Let $C$ be a closed affine subspace of $\mathcal{H}$. Then the following hold:
\begin{enumerate}
\item \label{prop:PR:Affine} The projector $\Pro_{C}$ and the reflector $\R_{C}$ are affine operators.
\item \label{prop:PR:Characterization} Let $x$ be in $\mathcal{H}$ and let $p$ be in $\mathcal{H}$. Then 
\begin{align*}
p=\Pro_{C} x \Longleftrightarrow p \in C \quad \mbox{and} \quad (\forall v \in C) ~ (\forall w \in C) \quad \innp{x-p,v-w}=0.
\end{align*}
\item \label{prop:PR:Pythagoras} $(\forall x \in \mathcal{H})$  $(\forall v \in C)$ $\norm{x-\Pro_{C} x }^{2} + \norm{v-\Pro_{C} x}^{2}=\norm{x-v }^{2}$.
\item \label{prop:PR:isometry}  $(\forall x \in \mathcal{H})$ $(\forall y \in \mathcal{H})$ $\norm{x-y}=\norm{\R_{C}x-\R_{C}y}$.
\item \label{prop:PR:ReflectorEquidist} $(\forall x \in \mathcal{H})$ $(\forall v \in C)$ $\norm{x-v}=\norm{\R_{C}x-v}$.
\end{enumerate}
\end{proposition}

\begin{proof}
\cref{prop:PR:Affine}: $\Pro_{C}$ is affine by {\rm \cite[Corollary~3.22(ii)]{BC2017}}; 
this implies that $\R_{C}=2\Pro_{C}- \Id$ is affine as well. 

\cref{prop:PR:Characterization}: {\rm \cite[Corollary~3.22(i)]{BC2017}}.

\cref{prop:PR:Pythagoras}: Indeed, for every $x \in \mathcal{H}$ and $v \in C$,
\begin{align*} 
\norm{x-v }^{2}&=\norm{x-\Pro_{C}x-(v-\Pro_{C}x)}^{2}\\
&=\norm{x-\Pro_{C} x }^{2} -2 \innp{x-\Pro_{C}x, v -\Pro_{C}x} + \norm{v - \Pro_{C} x}^{2}\\
&=\norm{x-\Pro_{C} x }^{2} + \norm{v - \Pro_{C} x }^{2}. \quad (\text{by \cref{prop:PR:Characterization}})
\end{align*}

\cref{prop:PR:isometry}: For every $x\in \mathcal{H}$, and for every $y \in \mathcal{H}$, by \cref{prop:PR:Characterization},
\begin{align*}
& \quad \quad ~~ \innp{\Pro_{C}x-\Pro_{C}y,\Pro_{C}x-x }- \innp{\Pro_{C}x-\Pro_{C}y,\Pro_{C}y-y }=0 \\
& \Longleftrightarrow \innp{\Pro_{C}x-\Pro_{C}y,\Pro_{C}x-\Pro_{C}y -(x-y) }=0 \\
& \Longleftrightarrow \norm{x-y}^{2} = 4 \norm{\Pro_{C}x-\Pro_{C}y }^{2} -4\innp{\Pro_{C}x-\Pro_{C}y,x-y } +\norm{x-y}^{2}\\
& \Longleftrightarrow \norm{x-y}^{2} = \norm{ (2 \Pro_{C}x-x)  -(2  \Pro_{C}y -y)}^{2}\\
& \Longleftrightarrow  \norm{x-y}=\norm{\R_{C}x-\R_{C}y}.  \quad (\text{by}~\R_{C}=2\Pro_{C}- \Id)
\end{align*}

\cref{prop:PR:ReflectorEquidist}:  Notice that $\Fix \R_{C} =C$ and then use \cref{prop:PR:isometry}.
\end{proof}

\subsection{Circumcenters} \label{Sec:Subsec:CircOpe}
In the whole subsection,
\begin{empheq}[box=\mybluebox]{equation*}
\mathcal{P}(\mathcal{H})~\text{is the set of all nonempty 
subsets of}~\mathcal{H}~\text{containing \emph{finitely many}
elements}.
\end{empheq}
By {\rm \cite[Proposition~3.3]{BOyW2018}}, we know that for every $K \in \mathcal{P}(\mathcal{H})$, there is at most one point $p \in \aff (K) $ such that $\{\norm{p-x} ~|~x \in K \}$ is a singleton. Hence, the following notion is well-defined.

\begin{definition}[circumcenter operator]  {\rm \cite[Definition~3.4]{BOyW2018}} \label{defn:Circumcenter}
The \emph{circumcenter operator} is 
\begin{empheq}[box=\mybluebox]{equation*}
\CCO{} \colon \mathcal{P}(\mathcal{H}) \to \mathcal{H} \cup \{ \varnothing \} \colon K \mapsto \begin{cases} p, \quad ~\text{if}~p \in \aff (K)~\text{and}~\{\norm{p-x} ~|~x \in K \}~\text{is a singleton};\\
\varnothing, \quad~ \text{otherwise}.
\end{cases}
\end{empheq}
In particular, when $\CCO(K) \in \mathcal{H}$, that is, $\CCO(K) \neq \varnothing$, we say that the circumcenter of $K$ exists and we call $\CCO(K)$ the circumcenter of $K$.
\end{definition}

\begin{fact}{\rm \cite[Example~3.6]{BOyW2018}}  \label{fact:CircForTwoPoints}
Let $x_1,x_2$ be in $\mathcal{H}$. 
Then 
\begin{align*}
\CCO{\big(\{x_1,x_2\}\big)}=\frac{x_{1} + x_{2}}{2}.
\end{align*} 
\end{fact}

\begin{fact} {\rm \cite[Theorem~4.1]{BOyW2018}} \label{fact:unique:LinIndpPformula}
Let $K=\{x_{1}, \ldots, x_{m}\} \in \mathcal{P}(\mathcal{H})$, where $x_{1}, \ldots, x_{m}$ are affinely independent.
Then $\CCO(K) \in \mathcal{H}$, which means that $\CCO(K)$ 
is the unique point satisfying the following two conditions:
\begin{enumerate}
\item \label{prop:unique:i} $\CCO(K) \in \aff (K)$, and
\item  \label{prop:unique:ii} $\{ \norm{\CCO(K)-y}~|~y \in K \}$ is a singleton.
\end{enumerate}
Moreover,
\begin{align*}
\CCO(K) = x_{1}+\frac{1}{2}(x_{2}-x_{1},\ldots,x_{m}-x_{1})
 G( x_{2}-x_{1},\ldots,x_{m}-x_{1})^{-1}
\begin{pmatrix}
 \norm{x_{2}-x_{1}}^{2} \\
 \vdots\\
\norm{x_{m}-x_{1}}^{2} \\
\end{pmatrix},
\end{align*}
where $G( x_{2}-x_{1},\ldots,x_{m-1}-x_{1},x_{m}-x_{1})$ is the
\emph{Gram matrix} of $x_{2}-x_{1},\ldots,x_{m-1}-x_{1},x_{m}-x_{1}$, i.e.,
\begin{align*}
&\quad ~ G( x_{2}-x_{1},\ldots, x_{m-1}-x_{1},x_{m}-x_{1})\\
&=
\begin{pmatrix} 
\norm{x_{2}-x_{1}}^{2} &\innp{x_{2}-x_{1},x_{3}-x_{1}} & \cdots & \innp{x_{2}-x_{1}, x_{m}-x_{1}}  \\ 
\vdots & \vdots & ~~& \vdots \\
\innp{x_{m-1}-x_{1},x_{2}-x_{1}} & \innp{x_{m-1}-x_{1}, x_{3}-x_{1}} & \cdots & \innp{x_{m-1}-x_{1},x_{m}-x_{1}} \\
\innp{x_{m}-x_{1},x_{2}-x_{1}} & \innp{x_{m}-x_{1},x_{3}-x_{1}} & \cdots & \norm{x_{m}-x_{1}}^{2} \\
\end{pmatrix}.
\end{align*}
\end{fact}

\begin{fact}  {\rm \cite[Theorem~8.1]{BOyW2018}} \label{fact:clform:three}
Suppose that 
$K=\{x,y,z\} \in \mathcal{P}(\mathcal{H})$ and that $\card (K) =3$. 
Then $x, y, z$ are affinely independent if and only if 
$\CCO(K) \in \mathcal{H}$. 
\end{fact}

Combining \cref{fact:CircForTwoPoints} and \cref{fact:clform:three}, we
obtain the following two results. 
\begin{corollary} \label{cor:x1x2x3Dom}
Let $K=\{x_{1},  x_{2}, x_{3}\} \in \mathcal{P}(\mathcal{H})$. Then $\CCO(K) \in \mathcal{H}$ if and only if exactly one of the following cases holds.
\begin{enumerate}
\item $\card \{x_{1}, x_{2},x_{3} \} \leq 2$.
\item $\card \{x_{1}, x_{2},x_{3}\}=3$ and if there is $\{\alpha, \beta\} \subseteq \mathbb{R}$ such that $\alpha (x_{2} -x_{1}) +\beta (x_{3}-x_{1})=0$, then $\alpha =0$ and $\beta =0$.
\end{enumerate} 
\end{corollary}

\begin{corollary} \label{cor:mathbbRcard3}
Let $a,b,c$ be in $\mathbb{R}$. Then there exists no $x \in \mathbb{R}$ such that $|x-a|=|x-b|=|x -c|$ if and only if $\card \{a,b,c\} =3$.
\end{corollary}

\begin{fact}[scalar multiples]  {\rm \cite[Proposition~6.1]{BOyW2018}} \label{fact:CircumHomoge}
Let $K \in \mathcal{P}(\mathcal{H})$ and $\lambda \in \mathbb{R} \smallsetminus \{0\}$. 
Then 
$\CCO(\lambda K)=\lambda \CCO(K)$. 
\end{fact}

\begin{fact}[translations]  {\rm \cite[Proposition~6.3]{BOyW2018}} \label{fact:CircumSubaddi}
Let $K \in \mathcal{P}(\mathcal{H})$ and $y \in \mathcal{H}$. Then
$\CCO(K+y)=\CCO(K)+y$.
\end{fact}

\begin{fact} {\rm \cite[Lemma~4.2]{BOyW2018}} \label{fact:unique:BasisPformula}
Let $K \in \mathcal{P}(\mathcal{H})$ and let $M \subseteq K$ be such that $\aff(M)=\aff(K)$. Suppose that $\CCO(K) \in \mathcal{H}$. 
Then $\CCO(K)=\CCO(M)$.
\end{fact}

\begin{fact}  {\rm \cite[Theorem~7.1]{BOyW2018}}  \label{fact:CCO:LimitCont}
Let $K=\{x_{1}, \ldots, x_{m}\} \in \mathcal{P}(\mathcal{H})$. Suppose that $\CCO(K) \in \mathcal{H}$. 
Then the following hold.
\begin{enumerate}
\item \label{prop:CCO:LimitCont:Linear}   
Set $t=\dim \big( \spn\{x_{2}-x_{1}, \ldots, x_{m}-x_{1}\} \big)$, 
and let $\widetilde{K}=\{x_{1}, x_{i_{1}}, \ldots, x_{i_{t}}\}
\subseteq K$ be such that $x_{i_{1}} -x_{1}, \ldots, x_{i_{t}}-x_{1}$ is a
basis of $\spn\{x_{2}-x_{1}, \ldots, x_{m}-x_{1}\}$.
Furthermore, let $\big( (x^{(k)}_{1}, x^{(k)}_{i_{1}}, \ldots,
x^{(k)}_{i_{t}}) \big)_{k \geq 1}$ $\subseteq$
$\mathcal{H}^{t+1}$ with $\lim_{k\rightarrow \infty}(
x^{(k)}_{1}, x^{(k)}_{i_{1}}, \ldots, x^{(k)}_{i_{t}})=(x_{1},
x_{i_{1}}, \ldots, x_{i_{t}})$, and 
set $(\forall k \geq 1)$ $\widetilde{K}^{(k)} = \{x^{(k)}_{1}, x^{(k)}_{i_{1}}, \ldots, x^{(k)}_{i_{t}}\}$. Then there exist $N \in \mathbb{N}$ such that for every $k \geq N$, $\CCO(\widetilde{K}^{(k)}) \in \mathcal{H}$ and
\begin{align*}
\lim_{k \rightarrow \infty} \CCO(\widetilde{K}^{(k)})= \CCO(\widetilde{K})=\CCO(K).
\end{align*}
\item  \label{prop:CCO:LimitCont:LinearIndep} 
Suppose that $ x_{1}, \ldots, x_{m-1}, x_{m}$ 
are affinely independent, 
and let $ \big( (x^{(k)}_{1}, \ldots, x^{(k)}_{m-1}, x^{(k)}_{m})
\big)_{k \geq 1}$ $\subseteq$  $\mathcal{H}^{m} $ satisfy
$\lim_{k\rightarrow \infty}( x^{(k)}_{1}, \ldots,x^{(k)}_{m-1},
x^{(k)}_{m})=(x_{1}, \ldots, x_{m-1},x_{m})$. Set $(\forall k \geq 1)$ $K^{(k)}=\{x^{(k)}_{1}, \ldots, x^{(k)}_{m-1}, x^{(k)}_{m}\}$. Then 
\begin{align*}
\lim_{k \rightarrow \infty} \CCO( K^{(k)} )= \CCO( K ).
\end{align*}
\end{enumerate}
\end{fact}

\begin{fact} {\rm \cite[Example~7.6]{BOyW2018}} \label{fact:CounterExampleContinuity}
Suppose that 
$\mathcal{H}=\mathbb{R}^{2}$. Let
$x_{1}=(-2,0)$ and $x_{2}=x_{3}=(2,0)$. Let $(\forall k \geq 1)$
$\big(x^{(k)}_{1}, x^{(k)}_{2},x^{(k)}_{3}\big)=\big( (-2,0),
(2,0),(2-\frac{1}{k},\frac{1}{4k})\big)$. Then
\begin{align*}
(\forall k \geq 1) \quad \CCO(\{x^{(k)}_{1}, x^{(k)}_{2},x^{(k)}_{3}\}) =
 \big(0, -8+\tfrac{2}{k}+\tfrac{1}{8k}\big).
 \end{align*}
\end{fact}

%---------------------------------------------Circumcenter mapping---------------------------------------------
\section{Circumcenter mappings induced by operators} \label{sec:CircumMapping}
Suppose that $T_{1}, \ldots, T_{m-1}, T_{m}$ are operators from $\mathcal{H}$ to $\mathcal{H}$, with $m \in \mathbb{N} \smallsetminus \{0\}$ and that 
\begin{empheq}[box=\mybluebox]{equation*}
\mathcal{S}=\{ T_{1}, \ldots, T_{m-1}, T_{m} \} \quad \text{and} \quad (\forall x \in \mathcal{H}) \quad \mathcal{S}(x)=\{ T_{1}x, \ldots, T_{m-1}x, T_{m}x\}.
\end{empheq}

\subsection{Definition}
\begin{definition}[induced circumcenter mapping] \label{def:cir:map}
The \emph{circumcenter mapping $\CC{\mathcal{S}}$ induced by $\mathcal{S}$} is 
\begin{empheq}[box=\mybluebox]{equation*}
\CC{\mathcal{S}} \colon \mathcal{H} \to \mathcal{H} \cup \{ \varnothing \} \colon x \mapsto \CCO(\mathcal{S}(x)),
\end{empheq}
that is, $\CC{\mathcal{S}} = \CCO{} \circ \mathcal{S}$. The \emph{domain} of $\CC{\mathcal{S}}$ is
\begin{align*}
\dom \CC{\mathcal{S}} = \{ x \in \mathcal{H} ~|~  \CC{\mathcal{S}}x \neq \varnothing \}.
\end{align*}
In particular, if $\dom \CC{\mathcal{S}} = \mathcal{H}$, then we say the
circumcenter mapping $\CC{\mathcal{S}}$ induced by $\mathcal{S}$, is
\emph{proper};
otherwise, we call $\CC{\mathcal{S}}$ \emph{improper}.
\end{definition}

\begin{remark} \label{rem:cir:map}
By Definitions \ref{def:cir:map} and \ref{defn:Circumcenter}, for every $x
\in \mathcal{H}$, if the circumcenter of the set $\mathcal{S}(x)$ defined in
\cref{defn:Circumcenter} does not exist in $\mathcal{H}$, then $\CC{\mathcal{S}}x= \varnothing
$. Otherwise, $\CC{\mathcal{S}}x$ is the unique point satisfying the two
conditions below:
\begin{enumerate}
\item $\CC{\mathcal{S}}x \in \aff(\mathcal{S}(x))=\aff\{T_{1}x, \ldots, T_{m-1}x, T_{m}x\}$, and 
\item $\norm{\CC{\mathcal{S}}x -T_{1}x}=\cdots =\norm{\CC{\mathcal{S}}x -T_{m-1}x}=\norm{\CC{\mathcal{S}}x -T_{m}x}$.
\end{enumerate}
\end{remark}

\subsection{Basic properties}
We start with some examples.

\begin{proposition} \label{prop:form:m2:Oper}
Assume $\mathcal{S}=\{T_{1}, T_{2}\}$. Then $\CC{\mathcal{S}}$ is proper. Moreover,
\begin{align*}
(\forall x \in \mathcal{H}) \quad  \CC{\mathcal{S}}x = \frac{T_{1}x +T_{2}x}{2}.
\end{align*}
\end{proposition}
\begin{proof}
Clear from \cref{fact:CircForTwoPoints} and \cref{def:cir:map}.
\end{proof}

\begin{corollary} \label{cor:T1T2T3Dom}
Let $\mathcal{S}=\{T_{1}, T_{2}, T_{3} \}$ and let $x \in \mathcal{H}$. Then $x \not \in \dom \CC{\mathcal{S}}$ if and only if 
 $\card \{T_{1}x, T_{2}x, T_{3}x \}=3$ and there exists $(\alpha, \beta) \in \mathbb{R}^{2} \smallsetminus \{(0,0)\}$ such that $\alpha (T_{2}x -T_{1}x) +\beta (T_{3}x -T_{1}x)=0$.
\end{corollary}
\begin{proof}
This follows from \cref{cor:x1x2x3Dom}.
\end{proof}

\begin{example}
Assume that $\mathcal{H}=\mathbb{R}^{2}$. 
Set $U_{1}=\mathbb{R}\cdot (1,0)$, $U_{2} = \mathbb{R} \cdot (0,1)$, and
let $\alpha \in \mathbb{R}$. Set $\mathcal{S} = \{\alpha \Id, R_{U_{1}}, R_{U_{2}}\}$. 
Then the following hold:
\begin{enumerate}
\item If $\alpha=0$, then $\dom \CC{\mathcal{S}} = \{(0,0)\}$.
\item If $\alpha=1$ or $\alpha=-1$, then $\dom \CC{\mathcal{S}} =\mathbb{R}^{2}$, i.e., $\CC{\mathcal{S}}$ is proper.
\item If $\alpha \in \mathbb{R} \smallsetminus \{0, 1, -1\}$, then $\dom \CC{\mathcal{S}} = \big( \mathbb{R}^{2} \smallsetminus (U_{1} \cup U_{2}) \big) \cup \{(0,0)\}$.
\end{enumerate}
\end{example}

\begin{proposition}\label{thm:CW:ExistWellDefined}
Suppose that for every $x \in \mathcal{H}$, there exists a point $p(x) \in
\mathcal{H}$ such that
\begin{enumerate}
\item  $p(x) \in \aff \{T_{1}x, \ldots, T_{m-1}x, T_{m}x\}$, and
\item  $\norm{p(x)-T_{1}x} =\cdots =\norm{p(x)-T_{m-1}x}=\norm{p(x)-T_{m}x}$.
\end{enumerate}
Then $\CC{\mathcal{S}}$ is proper and 
\begin{align*}
(\forall x \in \mathcal{H}) \quad \CC{\mathcal{S}}x=p(x).
\end{align*}
\end{proposition}
\begin{proof}
This follows from \cref{rem:cir:map}.
\end{proof}

\begin{proposition}\label{prop:CW:AffineIndepWellDefined}
Suppose that for every $x \in \mathcal{H}$, there exists $\I(x) \subseteq \I:=\{1,
\ldots, m\}$ such that $\card \big(\I(x) \big) = \card \big( \mathcal{S}(x)
\big)$ and $(T_{i}x)_{i \in \I(x) }$ is affinely independent. Then
$\CC{\mathcal{S}}$ is proper.
\end{proposition}
\begin{proof}
Let $x \in \mathcal{H}$.   Since $\I(x) \subseteq \I$, 
we have $\{ T_{i}x \}_{i \in \I(x) } \subseteq \mathcal{S}(x)$. 
The affine independence of $(T_{i}x)_{i \in \I(x) }$ yields $\card \big( \{ T_{i}x
\}_{i \in \I(x) } \big) = \card \big(\I(x) \big)$. Combining with $\card
\big(\I(x) \big) =\card \big( \mathcal{S}(x) \big)$, we obtain that $\{
T_{i}x \}_{i \in \I(x) } =\mathcal{S}(x)$, which implies that
\begin{align} \label{eq:prop:CW:AffineIndepWellDefined}
\CC{\mathcal{S}}x = \CCO{ \big( \mathcal{S}(x) \big)} = \CCO{} \big( \{
T_{i}x \}_{i \in \I(x) } \big).
\end{align}
Using the assumption that $(T_{i}x)_{i \in \I(x) }$ is affinely independent
again, by \cref{fact:unique:LinIndpPformula},
we deduce that $\CCO{} \big( \{ T_{i}x \}_{i \in \I(x) } \big) \in \mathcal{H} $. 
Combining with \cref{eq:prop:CW:AffineIndepWellDefined}, we deduce that
$(\forall x \in \mathcal{H})$ $\CC{\mathcal{S}}x \in \mathcal{H}$, i.e.,
$\CC{\mathcal{S}}$ is proper.
\end{proof}

The following example illustrates that the converse of
\cref{prop:CW:AffineIndepWellDefined} is not true in general. 

\begin{example} \label{exam:ProjecCCSProper}
Let $U$ be a closed linear subspace of $\mathcal{H}$ with $\{0\} \neq U
\varsubsetneqq \mathcal{H}$. Denote by $0$ also the zero operator: 
$(\forall x \in \mathcal{H})$ $0(x)=0$. Set $\mathcal{S}=\{\Id, \Pro_{U},
\Pro_{U^{\perp}}, 0\}$. Then the following hold:
\begin{enumerate}
\item \label{exam:ProjecCCSProper:0} $(\forall x \in \mathcal{H})$
$\CC{\mathcal{S}}x= \frac{x}{2}$; consequently, $\CC{\mathcal{S}}$ is proper.
\item \label{exam:ProjecCCSProper:ii} $\big(\forall x \in \mathcal{H} \smallsetminus ( U\cup U^{\perp}) \big)$ $x, \Pro_{U}x, \Pro_{U^{\perp}}x, 0 (x)$ are pairwise distinct.
\item \label{exam:ProjecCCSProper:iii} $(\forall x \in \mathcal{H})$ $\Id x, \Pro_{U}x,  \Pro_{U^{\perp}}x, 0( x)$ are affinely dependent.
\end{enumerate}
\end{example}
\begin{proof}
\cref{exam:ProjecCCSProper:0}: Let $x \in \mathcal{H}$. By \cref{prop:PR}\cref{prop:PR:Pythagoras} and by $0 \in U$ and $\Pro_{U}x \in U$,  we deduce that $\norm{\frac{x}{2}-\Pro_{U}\frac{x}{2} }^{2} +\norm{\Pro_{U}\frac{x}{2}}^{2} =\norm{ \frac{x}{2}}^{2}$ and that $\norm{\frac{x}{2}-\Pro_{U}\frac{x}{2} }^{2} +\norm{\Pro_{U}x-\Pro_{U}\frac{x}{2}}^{2} =\norm{ \frac{x}{2}-\Pro_{U}x}^{2}$. Combining with the linearity of $\Pro_{U}$, we obtain
\begin{align} \label{eq:exam:ProjecCCSProper:x}
\Norm{ \frac{x}{2}}=\Norm{ \frac{x}{2}-\Pro_{U}x}. 
\end{align}
Similarly, by  \cref{prop:PR}\cref{prop:PR:Pythagoras} again, replace $U$ in the above analysis by $U^{\perp}$ to yield that 
\begin{align} \label{eq:exam:ProjecCCSProper:x2}
\Norm{ \frac{x}{2}}=\Norm{ \frac{x}{2}-\Pro_{U^{\perp}}x}. 
\end{align}
Combining \cref{eq:exam:ProjecCCSProper:x} with \cref{eq:exam:ProjecCCSProper:x2}, we obtain that
\begin{align} \label{eq:exam:ProjecCCSProper:all}
\Norm{\frac{x}{2}}=\Norm{\frac{x}{2}-0(x)}=\Norm{\frac{x}{2}-x }=\Norm{ \frac{x}{2}-\Pro_{U}x} =\Norm{ \frac{x}{2}-\Pro_{U^{\perp}}x}.
\end{align}
Since $\frac{x}{2} = \frac{x}{2} + \frac{0}{2} \in \aff\{x, \Pro_{U}x, \Pro_{U^{\perp}}x, 0 (x)\}$, \cref{eq:exam:ProjecCCSProper:all} yields that $(\forall x \in \mathcal{H})$ $\CC{\mathcal{S}}x=\frac{x}{2}$.

\cref{exam:ProjecCCSProper:ii}: In fact, by \cref{MetrProSubs8}\cref{MetrProSubs8:iv},
\begin{subequations} \label{eq:exam:ProjecCCSProper:subeq}
\begin{align}
&x = \Pro_{U}x \Longleftrightarrow x \in U;\\
&x = \Pro_{U^{\perp}}x \Longleftrightarrow x \in U^{\perp};\\
&U \cap U^{\perp} =\{0\}.
\end{align}
\end{subequations}
In addition, Combining \cref{eq:exam:ProjecCCSProper:subeq}
 with \cref{MetrProSubs8}\cref{MetrProSubs8:ii}, we know that
\begin{align*}
\Pro_{U}x= \Pro_{U^{\perp}}x \Longrightarrow \Pro_{U}x= \Pro_{U^{\perp}}x=0 \Longrightarrow x=\Pro_{U}x+ \Pro_{U^{\perp}}x=0 \in U\cup U^{\perp}.
\end{align*}
Hence, for every $x \in \mathcal{H} \smallsetminus ( U\cup U^{\perp})$, $x, \Pro_{U}x, \Pro_{U^{\perp}}x, 0 (x)$ are pairwise distinct.

\cref{exam:ProjecCCSProper:iii}: Now for every $x \in \mathcal{H}$, 
\begin{align*}
& \quad ~~~  x=\Pro_{U}x+ \Pro_{U^{\perp}}x\\
& \Rightarrow x, \Pro_{U}x, \Pro_{U^{\perp}}x ~\text{are linear dependent}\\
& \Leftrightarrow  x -0, \Pro_{U}x-0, \Pro_{U^{\perp}}x-0 ~\text{are linear dependent}\\
& \Leftrightarrow   0( x), \Id x, \Pro_{U}x,  \Pro_{U^{\perp}}x  ~\text{are affinely dependent}. \quad (\text{by \cref{fac:AffinIndeLineInd}})
\end{align*}
The proof is complete.
\end{proof}

The following theorem provides a way to verify the properness of $\CC{\mathcal{S}}$ where $\mathcal{S}$ contains three operators.

\begin{theorem} \label{thm:Proper3}
Suppose that $\mathcal{S}=\{T_{1}, T_{2}, T_{3} \}$. Then $\CC{\mathcal{S}}$
is proper if and only if for every $x \in \mathcal{H}$ with $\card \big(
\mathcal{S}(x) \big) =3$, the vectors $T_{1}x, T_{2}x, T_{3}x $ are affinely independent.
\end{theorem}

\begin{proof}
By \cref{fact:clform:three}, for every $x \in \mathcal{H}$ with $\card \big( \mathcal{S}(x) \big) =3$,
\begin{align} \label{eq:thm:Proper3:AffinelyIndep}
\CC{\mathcal{S}}x \in \mathcal{H} \Longleftrightarrow T_{1}x, T_{2}x, T_{3}x
~\mbox{are affinely independent.}
\end{align}

\enquote{$\Longrightarrow$}: It follows directly from \cref{eq:thm:Proper3:AffinelyIndep}. 

\enquote{$\Longleftarrow$}: Assume that for every $x \in \mathcal{H}$ with $\card \big( \mathcal{S}(x)\big)=3$, $T_{1}x, T_{2}x, T_{3}x $ are affinely independent in $\mathcal{H}$. Let $x \in \mathcal{H}$. If $\card \big( \mathcal{S}(x)\big)=3$, by \cref{eq:thm:Proper3:AffinelyIndep} and the assumption,  then $\CC{\mathcal{S}}x \in \mathcal{H}$.  Assume $\card \big( \mathcal{S}(x)\big) \leq 2$, by \cref{prop:form:m2:Oper}, $\CC{\mathcal{S}}x \in \mathcal{H}$. Altogether, $(\forall x \in \mathcal{H})$, $\CC{\mathcal{S}}x \in \mathcal{H} $, which means that $\CC{\mathcal{S}}$ is proper.
\end{proof}

\begin{proposition} \label{prop:CW:FixPointSet}   
Suppose that $\mathcal{S}=\{ T_{1}, \ldots, T_{m-1}, T_{m} \}$. 
Then the following hold:
\begin{enumerate}
\item \label{prop:CW:FixPointSet:inclu} $\cap^{m}_{i=1} \Fix T_{i} \subseteq \Fix \CC{\mathcal{S}}$.
\item \label{prop:CW:FixPointSet:equa} 
If $\Fix \CC{\mathcal{S}} \subseteq \cup^{m}_{i=1} \Fix T_{i}$, 
then $\Fix \CC{\mathcal{S}} =\cap^{m}_{i=1} \Fix T_{i} $. 
\item \label{prop:Id:FixPointSet:equa} If $T_{1} =\Id$, 
then $\cap^{m}_{i=1} \Fix T_{i} =\Fix \CC{\mathcal{S}}$
\end{enumerate}
\end{proposition}

\begin{proof}
\cref{prop:CW:FixPointSet:inclu}: Let $x \in \cap^{m}_{i=1} \Fix T_{i} $. Then 
\begin{align} \label{eq:prop:CW:FixPointSet:fx}
(\forall i \in \{1, \ldots, m-1,m\}) \quad T_{i}x=x,
\end{align}
which yields that $\aff \{T_{1}x, \ldots, T_{m-1}x,T_{m}x\}=\aff \{x\}=\{x\}$. In addition, by \cref{eq:prop:CW:FixPointSet:fx},
\begin{align*}
\norm{x-T_{1}x}=\cdots=\norm{x-T_{m-1}x}=\norm{x-T_{m}x}=0.
\end{align*}
Therefore, we obtain that $\CC{\mathcal{S}}x=x$, which means that $x \in \Fix \CC{\mathcal{S}}$. Hence, $\cap^{m}_{i=1} \Fix T_{i} \subseteq \Fix \CC{\mathcal{S}}$.

\cref{prop:CW:FixPointSet:equa}: Let $x \in \Fix \CC{\mathcal{S}}$. 
By the assumption, there is $i_{0} \in \{1,\ldots,m\}$ such that 
\begin{align}\label{eq:i0:prop:CW:FixPointSet}  
x = T_{i_{0}}x
\end{align}
Now $x \in \Fix \CC{\mathcal{S}}$, i.e., $x=\CC{\mathcal{S}}x$, implies that 
\begin{align} \label{eq:Tx:prop:CW:FixPointSet}  
\norm{x-T_{1}x}=\cdots=\norm{x-T_{m-1}x}=\norm{x-T_{m}x}.
\end{align}
Combining \cref{eq:Tx:prop:CW:FixPointSet} with \cref{eq:i0:prop:CW:FixPointSet}, we obtain that 
\begin{align*}
\norm{x-T_{1}x}=\cdots=\norm{x-T_{m-1}x}=\norm{x-T_{m}x}=0,
\end{align*}
which means that $x \in \cap^{m}_{i=1} \Fix T_{i}$. Hence, $\Fix \CC{\mathcal{S}} \subseteq \cap^{m}_{i=1} \Fix T_{i}$. Combining with \cref{prop:CW:FixPointSet:inclu}, we deduce that $\Fix \CC{\mathcal{S}} =\cap^{m}_{i=1} \Fix T_{i} $.

\cref{prop:Id:FixPointSet:equa}: 
If $T_1=\Id$, then $\Fix T_1 = \mathcal{H}$ and the result follows from
\cref{prop:CW:FixPointSet:equa}. 
\end{proof}

\begin{example}
Assume that $\mathcal{H} =\mathbb{R}^{2}$. Set 
$T_{1} = \Pro_{\mathbf{B}[(-2,0);1]}$, $T_{2} =
\Pro_{\mathbf{B}[(0,2);1]}$, $T_{3} = \Pro_{\mathbf{B}[(2,0);1]}$, and 
$\mathcal{S} = \{T_{1}, T_{2},T_{3}\}$. Then $\CC{\mathcal{S}}$ is proper.
Moreover, $\varnothing = \cap^{3}_{i=1} \Fix T_{i} \subsetneqq \Fix \CC{\mathcal{S}}
= \{(0,0)\} $.
\end{example}
\begin{proof}
The properness of $\CC{\mathcal{S}}$ follows from \cref{thm:Proper3} while the rest 
is a consequence of elementary manipulations. 
\end{proof}
%---------------------------------------------------------------------------------------------------------------------------------------
%---------------------------------------------Continuity of circumcenter mapping---------------------------------------------
%---------------------------------------------------------------------------------------------------------------------------------------

\subsection{Continuity} \label{subsec:ContinuCircumMapping}
\begin{proposition} \label{prop:OperContIndep}
Assume that the elements of $\mathcal{S}=\{T_{1}, \ldots, T_{m-1}, T_{m}\}$
are continuous operators and that  $x \in \dom \CC{\mathcal{S}}$. 
Then the following hold: 
\begin{enumerate}
\item \label{prop:OperContIndep:i} Let $\widetilde{\mathcal{S}}_{x}=\{T_{1},  T_{i_{1}}, \ldots, T_{i_{d_{x}}} \} \subseteq \mathcal{S}$ be such that\footnotemark ~
$T_{i_{1}}x-T_{1}x, \ldots, T_{i_{d_{x}}} x-T_{1}x$  is a basis of $\spn\{T_{2}x-T_{1}x, \ldots, T_{m}x-T_{1}x \}$ .
Then for every $(x^{(k)})_{k\in \mathbb{N}}\subseteq \mathcal{H}$ satisfying $\lim_{k \rightarrow \infty} x^{(k)}=x$, there exists $N \in \mathbb{N}$ such that for every $k \geq N$, $\CC{\widetilde{\mathcal{S}}_{x}}(x^{(k)}) \in \mathcal{H}$. Moreover 
\begin{align} \label{eq:prop:OperContIndep:first}
\lim_{k \rightarrow \infty}\CC{\widetilde{\mathcal{S}}_{x}}(x^{(k)})=\CC{\widetilde{\mathcal{S}}_{x}}x=\CC{\mathcal{S}}x.
\end{align}
\item \label{prop:OperContIndep:ii}
If $T_{1}x, \ldots, T_{m-1}x, T_{m}x$ are affinely independent, then $\CC{\mathcal{S}}$ is continuous at $x$.
\end{enumerate}
\end{proposition}
\footnotetext{When $\mathcal{S}(x)$ is a singleton, then
$\widetilde{\mathcal{S}}_{x}= \{T_{1}\}$ by the standard convention that $
\varnothing $ is the basis of $\{0\}$.}
\begin{proof}
\cref{prop:OperContIndep:i}: Let $(x^{(k)})_{k\in \mathbb{N}}\subseteq \mathcal{H}$ satisfying $\lim_{k \rightarrow \infty} x^{(k)}=x$. Now 
\begin{align*}
&\mathcal{S}= \{ T_{1}, \ldots, T_{m-1}, T_{m}\},\quad \mathcal{S}(x)=\{T_{1}(x), \ldots, T_{m-1}x, T_{m}x\},\\
&\widetilde{\mathcal{S}}_{x} = \{ T_{1}, T_{i_{1}}, \ldots, T_{i_{d_{x}}}\}, \quad 
\widetilde{\mathcal{S}}_{x}(x)=\{T_{1}x,  T_{i_{1}}x, \ldots, T_{i_{d_{x}}}x \},
\quad \widetilde{\mathcal{S}}_{x}(x^{(k)})=\{T_{1}x^{(k)},  T_{i_{1}}x^{(k)}, \ldots, T_{i_{d_{x}}}x^{(k)} \}.
\end{align*}
By \cref{def:cir:map}, $\CC{\mathcal{S}}x \in \mathcal{H}$ means that $\CCO(\mathcal{S}(x))\in \mathcal{H}$.
By assumptions, 
\begin{align*}
T_{i_{1}}x-T_{1}x, \ldots, T_{i_{d_{x}}} x-T_{1}x ~\text{ is a basis of }~\spn \{T_{2}x-T_{1}x, \ldots, T_{m}x-T_{1}x \} . 
\end{align*}
Substituting the $K$, $\widetilde{K}$ and $\widetilde{K}^{(k)}$ in \cref{fact:CCO:LimitCont}\cref{prop:CCO:LimitCont:Linear} by the above $\mathcal{S}(x)$, $\widetilde{\mathcal{S}}_{x}(x)$ and $\widetilde{\mathcal{S}}_{x}(x^{(k)})$ respectively, we obtain the desired results.

\cref{prop:OperContIndep:ii}: This follows easily from \cref{prop:OperContIndep:i}. 
\end{proof}

The next result summarizes conditions under which the proper circumcenter mapping 
$\CC{\mathcal{S}}$ is continuous at a point $x$.

\begin{proposition} \label{prop:CircumMapContinu:Allcase}
Assume that the elements of $\mathcal{S}=\{T_{1}, \ldots, T_{m-1}, T_{m}\}$
are continuous operators and that $\CC{\mathcal{S}}$ is proper.
Let $x \in \mathcal{H}$. The following assertions hold:
\begin{enumerate}
\item \label{thm:CircumMapContinu:Allcase:affineIndep} If $T_{1}x, \ldots, T_{m-1}x, T_{m}x$ are affinely independent, then $\CC{\mathcal{S}}$ is continuous at $x$.
\item \label{thm:CircumMapContinu:Allcase:affineDep} 
If $T_{1}x, \ldots, T_{m-1}x, T_{m}x$ are affinely dependent and
$m\leq 2$, then $\CC{\mathcal{S}}$ is continuous at $x$. 
\end{enumerate}
\end{proposition}
\begin{proof}
\cref{thm:CircumMapContinu:Allcase:affineIndep} follows from
\cref{prop:OperContIndep}\cref{prop:OperContIndep:ii} while 
\cref{thm:CircumMapContinu:Allcase:affineDep} is a consequence of \cref{prop:form:m2:Oper}.
\end{proof}

The following examples show that even when $T_{1}x, \ldots, T_{m-1}x, T_{m}x$ are
affinely dependent and $m \geq 3$, then $\CC{S}$ may still be continuous at $x$.

\begin{example}  \label{exam:ProperCirMapContinuous:UUperp}
Suppose that $U$ is a closed linear subspace of $\mathcal{H}$ such that
 $\{0\}\subsetneqq U \varsubsetneqq \mathcal{H}$. 
Set $\mathcal{S}=\{\Id, \R_{U}, \R_{U^{\perp}}\}$. 
Then the following hold: 
\begin{enumerate}
\item \label{exam:ProperCirMapContinuous:UUperp:i} 
The vectors $x,\R_{U}x,\R_{U^\perp}x$ are affinely dependent for 
every $x\in U \cup U^\perp$. 
\item \label{exam:ProperCirMapContinuous:UUperp:ii} $\CC{S}\equiv 0$ which is thus 
proper and continuous on $\mathcal{H}$. 
\end{enumerate}
\end{example} 

\begin{proof}
\cref{exam:ProperCirMapContinuous:UUperp:i}: For every $x \in U$ (respectively $x \in U^{\perp}$), $\R_{U}x=x $ (respectively $\R_{U^{\perp}}x=x$), which implies that $x, \R_{U}x, \R_{U^{\perp}}x$, which is $x, x, \R_{U^{\perp}}x $ (respectively $x, \R_{U}x, x$) are affinely dependent.

\cref{exam:ProperCirMapContinuous:UUperp:ii}: Since
$\Id=\Pro_{U}+\Pro_{U^{\perp}}$ and $\R_{U}=2\Pro_{U}-\Id $, we have 
\begin{align*}
\frac{\R_{U}+\R_{U^{\perp}}}{2} =\frac{(2\Pro_{U}- \Id)+(2\Pro_{U^{\perp}} - \Id)}{2} =\frac{1}{2}\Big(2\Pro_{U} - \Id+2(\Id-\Pro_{U})-\Id \Big)=0.
\end{align*}
Let $x \in \mathcal{H}$. Then $0 =\frac{\R_{U}x+\R_{U^{\perp}}x }{2} \in \aff \{x ,\R_{U}x, \R_{U^{\perp}}x\}$. In addition, clearly $0 \in U \cap U^{\perp}$. In \cref{prop:PR}\cref{prop:PR:ReflectorEquidist}, substitute $C =U$, and let the point $v=0$. We get  $\norm{x}=\norm{\R_{U}x}$. Similarly,  In \cref{prop:PR}\cref{prop:PR:ReflectorEquidist},  substitute  $C=U^{\perp}$  and let the point $v=0$ . We get $\norm{x}=\norm{\R_{U^{\perp}}x}$. Hence, we have
\begin{enumerate}[label=(\alph*)]
\item $0 \in \aff \{x ,\R_{U}x, \R_{U^{\perp}}x\}$ and
\item $\norm{0-x}=\norm{ 0-\R_{U}x}=\norm{0-\R_{U^{\perp}}x} $,
\end{enumerate}
which means that  $(\forall x \in \mathcal{H})$ $\CC{S}(x)=0$. 
\end{proof}

\begin{example} \label{exam:ProperCirMapContinuous}
Assume that $\mathcal{H}=\mathbb{R}^{2}$ and $\mathcal{S}=\{T_{1},T_{2},T_{3}\}$, 
where for every $(x,y) \in \mathbb{R}^{2}$,
\begin{align*}
T_{1}(x,y)=(x,y);\quad
T_{2}(x,y)=(-x,y);\quad
T_{3}(x,y)=\big(x,-\tfrac{1}{4}(x-2)\big).
\end{align*}
Then 
\begin{enumerate}
\item \label{exam:ProperCirMapContinuous:i}
$T_{1}(x,y),T_{2}(x,y),T_{3}(x,y)$ are affinely independent if and only if
$2x \big( -\tfrac{1}{4}(x-2)-y\big) \neq 0$;
\item \label{exam:ProperCirMapContinuous:ii} $\big(\forall (x,y) \in
\mathbb{R}^{2}\big) \quad \CC{\mathcal{S}} (x,y)=\big(0, \frac{1}{2}\big(y
-\frac{1}{4}(x-2)\big)\big)$. 
\end{enumerate}
Consequently, $\CC{\mathcal{S}}$ is proper and continuous. 
\end{example}

The following example shows that even if the operators in $\mathcal{S}$ are continuous,
we generally have 
\begin{align*}
\CC{\mathcal{S}} ~\text{is proper}  \nRightarrow \CC{\mathcal{S}} ~\text{is continuous}.  
\end{align*}

\begin{example} \label{exam:ProperCirMapDiscontinuous}
Assume that $\mathcal{H}=\mathbb{R}^{2}$ and $\mathcal{S}= \{T_{1}, T_{2}, T_{3}\}$,
where for every $(x,y) \in \mathbb{R}^{2}$,
\begin{align*}
T_{1}(x,y)=(2,0);\quad
T_{2}(x,y)=(-2,0);\quad
T_{3}(x,y)=\big(x,-\tfrac{1}{4}(x-2)\big).
\end{align*}
Then 
\begin{enumerate}
 \item \label{exam:ProperCirMapDiscontinuous:i}$\CC{\mathcal{S}}$ is proper; 
 \item \label{exam:ProperCirMapDiscontinuous:ii} Let $(\forall k \geq 1)$ $(x^{(k)},y^{(k)})=(2-\frac{1}{k},0)$. Then $ \lim_{k \rightarrow \infty} \CC{\mathcal{S}}(x^{(k)},y^{(k)}) =(0,-8)\neq (0,0)=\CC{\mathcal{S}}(2,0)$. Consequently,   
 $\CC{\mathcal{S}}$ is not continuous at the point $(2,0)$.
 \end{enumerate}
\end{example} 

\begin{proof}
\cref{exam:ProperCirMapDiscontinuous:i}: Let $(x,y) \in \mathbb{R}^{2}$. Now by \cref{fac:AffinIndeLineInd}, 
\begin{align*}
&T_{1}(x,y),T_{2}(x,y),T_{3}(x,y)~\text{are affinely independent}\\
\Longleftrightarrow & T_{2}(x,y)- T_{1}(x,y),T_{3}(x,y)-T_{1}(x,y) ~\text{are linearly independent} \\
\Longleftrightarrow & (-4,0), \big(x-2, -\tfrac{1}{4}(x-2) \big) ~\text{are linearly independent}\\
 \Longleftrightarrow & \det(A) \neq 0, ~\text{where}~A= \begin{pmatrix} 
-4 & x-2  \\ 
0 & -\frac{1}{4}(x-2) \\
\end{pmatrix}\\
\Longleftrightarrow & x-2 \neq 0.
\end{align*}
Hence, by \cref{cor:x1x2x3Dom}, when $x-2 \neq 0$, we have $\CC{\mathcal{S}}(x,y)\in \mathcal{H}$.
Actually, when $x-2=0$, that is $x=2$, then for every $y \in \mathbb{R}$,
\begin{align*}
T_{1}(2,y)=(2,0), T_{2}(2,y)=(-2,0),T_{3}(2,y)=\big(2,-\tfrac{1}{4}(2-2)\big)=(2,0).
\end{align*}
By \cref{prop:form:m2:Oper}, we know that  $\CC{\mathcal{S}}(x,y)=(0,0)\in \mathcal{H}$. Hence,  $\CC{\mathcal{S}}$ is proper.

\cref{exam:ProperCirMapDiscontinuous:i}: Let $(\overline{x},\overline{y})=(2,0)$, and $(\forall k \geq 1)$ $(x^{(k)},y^{(k)})=(2-\frac{1}{k},0)$. By the analysis in \cref{exam:ProperCirMapDiscontinuous:i} above, we know
\begin{align} \label{eq:exam:barxValue}
\CC{\mathcal{S}}(\overline{x},\overline{y})=(0,0).
\end{align}
On the other hand, since 
\begin{align*}
\mathcal{S}(x^{(k)},y^{(k)})=\Big\{ T_{1}(x^{(k)},y^{(k)}),
T_{2}(x^{(k)},y^{(k)}),T_{3}(x^{(k)},y^{(k)}) \Big\} =\Big\{ (2,0), (-2,0),
\big(2-\tfrac{1}{k},\tfrac{1}{4k}\big) \Big\},
\end{align*} 
and since, by \cref{def:cir:map}, $ \CC{\mathcal{S}}(x^{(k)},y^{(k)})= \CCO(\mathcal{S}(x^{(k)},y^{(k)}) ) $,  we deduce that,  by \cref{fact:CounterExampleContinuity},
\begin{align} \label{eq:exam:xkValue}
\CC{\mathcal{S}}(x^{(k)},y^{(k)}) =\big(0, -8+\tfrac{2}{k} +\tfrac{1}{8k}\big).
\end{align}
Hence, 
\begin{align*}
\lim_{k \rightarrow \infty} \CC{\mathcal{S}}(x^{(k)},y^{(k)}) =(0,-8) \neq (0,0) \stackrel{\cref{eq:exam:barxValue}}{=} \CC{\mathcal{S}}(2,0)
\end{align*}
and we are done. 
\end{proof}

%---------------------------------------------About demiclosedness principle---------------------------------------------
%-----------------------------------------------------------------------------------------------------------------------------------------------------
%----------------------------------------------------------------------------------------------------------------------------------------------------

\subsection{The Demiclosedness Principle for circumcenter mappings} \label{subsec:Demiclosedness}

Let $T\colon \HH\to \HH$ be nonexpansive. 
Then 
\begin{equation}
  \label{e:demiclosed}
\left. 
\begin{array}{c} x_k\weakly x \\
  x_k-Tx_k\to 0
\end{array}
  \right\}
\;\;  \Rightarrow  \;\;
x\in\Fix T. 
\end{equation}
This well known implication (see \cite[Theorem~3(a)]{B1968}) 
is \emph{Browder's Demiclosedness Principle}; it is a powerful tool in the 
study of nonexpansive mappings.
(Technically speaking, \eqref{e:demiclosed} states that 
$\Id-T$ is demiclosed at $0$, but because a shift of a nonexpansive
mapping is still nonexpansive, it is demiclosed everywhere.)
For the sake of brevity, we shall simply say that 
\begin{center}
\enquote{the demiclosedness principle holds for $T$}  whenever
\eqref{e:demiclosed} holds.
\end{center}
Clearly, the demiclosedness principle holds
whenever $T$ is weak-to-strong continuous, which is
the case when $T$ is continuous and $\HH$ is finite-dimensional.
The demiclosedness principle also holds for 
so-called subgradient projectors; see \cite[Lemma~5.1]{BCW2015} for details.

We now obtain a condition sufficient for the circumcenter mapping to satisfy
the demiclosedness principle.
Throughout, we assume $T_1,\ldots,T_m$ are mappings from $\HH$ to $\HH$. 

\begin{theorem} \label{thm:Main:Demi}
 Suppose that the demiclosedness principle holds
 for each element in 
$\mathcal{S} = \{T_{1}, T_{2}, \ldots, T_{m}\}$. 
In addition, assume that $\CC{\mathcal{S}}$ is proper and that the implication 
\begin{align} \label{eq:assum:thm:Main:Demi} 
x_{k} - \CC{\mathcal{S}}x_{k}
\rightarrow 0 \;\;\Rightarrow\;\; (\forall i \in \{1, \ldots, m\})
~\CC{\mathcal{S}}x_{k} -T_{i}x_{k} \rightarrow 0 \end{align}
holds.
Then the demiclosedness principle holds for  $\CC{\mathcal{S}}$ and
 $\Fix \CC{\mathcal{S}} = \bigcap^{m}_{i=1}\Fix T_{i}$.
\end{theorem}
\begin{proof}
Let $x_{k} \rightharpoonup x $ and 
\begin{align} \label{eq:assum:T:thm:Main:Demi}
 x_{k}- \CC{\mathcal{S}} x_{k} \rightarrow 0.
\end{align}
By \cref{eq:assum:T:thm:Main:Demi} and \cref{eq:assum:thm:Main:Demi},
\begin{align}  \label{eq:Ti:thm:Main:Demi}
(\forall i \in \{1, \ldots, m\}) \quad \CC{\mathcal{S}}x_{k} -T_{i}x_{k} \rightarrow 0.
\end{align}
Hence, 
\begin{align}
(\forall i \in \{1, \ldots, m\})  \quad \norm{x_{k} - T_{i}x_{k} } \leq \norm{x_{k} -\CC{\mathcal{S}} x_{k} } + \norm{ \CC{\mathcal{S}} x_{k} -T_{i}x_{k}} \rightarrow 0.
\end{align}
Because the demiclosedness principle holds
for each $T_{i}$, 
we deduce that $x \in \cap^{m}_{i=1}\Fix T_{i} \subseteq \Fix
\CC{\mathcal{S}}$, where the last inclusion follows from
\cref{prop:CW:FixPointSet}\cref{prop:CW:FixPointSet:inclu}. 
Therefore, $x
-\CC{\mathcal{S}}x =0 $, which shows that the demiclosedness principle 
holds for $\CC{\mathcal{S}}$. 
To verify the remaining assertion, let $\bar{x} \in \Fix \CC{\mathcal{S}}$.
For every $k \in \mathbb{N}$, substitute $x_{k}$ by $\bar{x}$. Then using
the assumption \cref{eq:assum:thm:Main:Demi}, we deduce that $(\forall i \in
\{1, \ldots, m\})$ $\CC{\mathcal{S}}\bar{x} -T_{i}\bar{x}=0$. Combining with
$(\forall i \in \{1, \ldots, m\})$ $\norm{\bar{x} - T_{i}\bar{x} } \leq
\norm{\bar{x} -\CC{\mathcal{S}} \bar{x} } + \norm{ \CC{\mathcal{S}} \bar{x}
-T_{i}\bar{x}}$, we obtain that $\bar{x} \in \cap^{m}_{i=1}\Fix T_{i}$.
Hence, $\Fix \CC{\mathcal{S}} \subseteq \cap^{m}_{i=1}\Fix T_{i}$. Therefore,
the desired result follows from
\cref{prop:CW:FixPointSet}\cref{prop:CW:FixPointSet:inclu}.
\end{proof}

\begin{corollary} \label{cor:Main:Demi:Assum}
Suppose that $T_{1} =\Id$ and that $\CC{\mathcal{S}}$ is proper. Then
the implication 
\begin{align} \label{eq:cor:assum:Main:Demi}
x_{k} - \CC{\mathcal{S}}x_{k}  \rightarrow  0  \;\;\Rightarrow \;\;
(\forall i \in \{1, \ldots, m\}) ~\CC{\mathcal{S}}x_{k} -T_{i}x_{k} \rightarrow 0.
\end{align} 
holds. 
\end{corollary}
\begin{proof}
Since $\CC{\mathcal{S}}$ is proper, by \cref{rem:cir:map},
$\norm{\CC{\mathcal{S}}x_{k} -x_{k} } = \norm{\CC{\mathcal{S}}x_{k} -
T_{2}x_{k} }=\cdots = \norm{\CC{\mathcal{S}}x_{k} - T_{m}x_{k} }$,
which implies that \cref{eq:cor:assum:Main:Demi} is true.
\end{proof}

\begin{proposition} \label{prop:demi:IdId}
Suppose that $T_{1} =\Id$, that  for every $i \in \{2,\ldots,m\}$, 
the demiclosedness principle holds for $T_{i}$, 
that $\mathcal{S} = \{T_{1}, T_{2}, \ldots, T_{m}\}$,
and that $\CC{\mathcal{S}}$ is proper. 
Then the demiclosedness principle holds for $\CC{\mathcal{S}}$ and 
$\Fix \CC{\mathcal{S}} = \cap^{m}_{i=1}\Fix
T_{i}$. 
\end{proposition}
\begin{proof}
Combine \cref{thm:Main:Demi} with \cref{cor:Main:Demi:Assum}.
\end{proof}

We are now ready for the main result of this section. 

\begin{theorem}[a demiclosedness principle for circumcenter mappings] \label{prop:demi:Nonexpan}
Suppose that $T_{1} =\Id$, that each operator in 
$\mathcal{S} = \{T_{1}, T_{2}, \ldots, T_{m}\}$ is nonexpansive, 
and that 
$\CC{\mathcal{S}}$ is proper. Then 
the demiclosedness principle holds for $\CC{\mathcal{S}}$ and 
$\Fix \CC{\mathcal{S}} = \cap^{m}_{i=1}\Fix T_{i}$.
\end{theorem}
\begin{proof}
Combine Browder's Demiclosedness Principle with \cref{prop:demi:IdId}.
\end{proof}

We now present (omitting its easy proof) another consequence of \cref{prop:demi:IdId}. 

\begin{corollary} \label{prop:FiniteIdresult}
Suppose that $\mathcal{H}$ is finite-dimensional, 
that $\mathcal{S} = \{T_{1}, T_{2}, \ldots, T_{m}\}$, 
where $T_{1} = \Id$ and $T_j$ is continuous for every $j \in \{2, \ldots,
m\}$, and that $\CC{\mathcal{S}}$ is proper. Then
the demiclosedness principle holds for $\CC{\mathcal{S}}$. 
In particular, 
\begin{align} 
\label{eq:prop:sadresult:assum}
\left.
\begin{array}{c}
x_{k} \to \overline{x}\\
x_{k} - \CC{\mathcal{S}}x_{k} \to 0
\end{array}
\right\}
 \;\; \Rightarrow \;\;
\overline{x} \in \cap^{m}_{j=1}\Fix T_{j} = \Fix \CC{\mathcal{S}}.
\end{align}
\end{corollary}

We now provide an example where 
the demiclosedness principle does not hold for $\CC{\mathcal{S}}$.

\begin{example}  \label{exam:DemicloseTRUE3}
Suppose that $\mathcal{H} =\mathbb{R}^{2}$. Set $L=\big\{(u,v)\in\HH
~\big|~v=-\frac{1}{4}u+\frac{1}{2}\big\}$. 
Assume that $\mathcal{S} = \{T_{1}, T_{2}, T_{3}\}$, where
\begin{align*}
\big(\forall (u,v) \in \HH \big) \quad 
T_{1}(u,v) =(-2,0),\; T_{2}(u,v) =(2,0) \quad \text{and} \quad T_{3}(u,v) = \Pro_{L}(u,v).
\end{align*}
Set $\overline{x} =(0,-8)$ and 
$(\forall k \in \mathbb{N} \smallsetminus \{0\})$ 
$x_{k}=\big(\tfrac{1}{k},- \tfrac{1}{4k}-8\big)$.  
Then the following hold. 
 \begin{enumerate}
 \item \label{exam:DemicloseTRUE3:proper} $\CC{\mathcal{S}}$ is proper.
 \item \label{exam:DemicloseTRUE3:fix} $\Fix \CC{\mathcal{S}} =\varnothing$.
 \item \label{exam:DemicloseTRUE3:converg} 
 $\displaystyle \lim_{k \rightarrow \infty}\CC{\mathcal{S}}x_{k} = \overline{x} =\lim_{k
 \rightarrow \infty} x_{k} $; consequently, 
 $\displaystyle \lim_{k \rightarrow \infty}
 (x_{k} - \CC{\mathcal{S}}x_{k}) =0$. (See also \cref{Demiconti.png}.) 
  \item \label{exam:DemicloseTRUE3:overlinex} 
$\overline{x} \not\in \Fix \CC{\mathcal{S}}$; 
consequently, the demiclosedness principle does not hold
for $\CC{\mathcal{S}}$. 
 \end{enumerate}
\end{example}

\begin{proof}
\cref{exam:DemicloseTRUE3:proper}: Let $x \in \mathcal{H} $. If $ T_{3}x \in
\mathbb{R}\cdot(1,0)$, then $T_{3}x=(2,0)$ and so $ \CC{\mathcal{S}}x
=(0,0)$. Now assume that $ T_{3}x \not\in \mathbb{R}\cdot(1,0)$. Then $T_{1}x, T_{2}x,
T_{3}x $ are affinely independent. 
Hence, by \cref{thm:Proper3}, $ \CC{\mathcal{S}}x \in \mathcal{H}$. 
Altogether, $\CC{\mathcal{S}}$ is proper.
 
 \cref{exam:DemicloseTRUE3:fix}: Since $T_{1}x=(-2,0)$ and $T_{2}x=(2,0)$, by
definition of circumcenter mapping,
\begin{align*}
\CC{\mathcal{S}}x \in \mathbb{R} \cdot (0,1),
\end{align*}
which implies if $x \in \Fix \CC{\mathcal{S}}$, then $x \in \mathbb{R} \cdot
(0,1)$. Since $T_{3}(0,-8) =
 \Pro_{L}(0,-8) = (2,0)$, by \cref{prop:form:m2:Oper}, $
 \CC{\mathcal{S}}(0,-8) =(0,0) \neq (0,-8)$. Hence,
$(0,-8) \not\in
\Fix \CC{\mathcal{S}} $. Let $x:=(0,v) \in \big( \mathbb{R} \cdot (0,1)
\big)\smallsetminus \{(0,-8)\}$. As seen in the proof of
\cref{exam:DemicloseTRUE3:proper}, the vectors $T_{1}x, T_{2}x, T_{3}x $ are affinely
independent.
Hence, by definition of circumcenter mapping,  in this case 
\begin{align*}
\CC{\mathcal{S}}x ~\text{is the intersection of }~\mathbb{R} \cdot (0,1)  ~\text{and the perpendicular bisector of the two points}~T_{2}x, T_{3}x.
\end{align*}
Denote by $\CC{\mathcal{S}}x:= (0,w)$. Some easy calculation yields that if $v > -8$, then $w >v$; if $v <-8$, then $w <v$, which means that $\CC{\mathcal{S}}x \neq x $. Altogether, $\Fix \CC{\mathcal{S}} =\varnothing$.
 
 \cref{exam:DemicloseTRUE3:converg}: Let $k \in \mathbb{N} \smallsetminus
 \{0\}$. Since $x_{k}=\big(\tfrac{1}{k},- \tfrac{1}{4k}-8\big)$, 
 by definition of $T_{3}$,
 \begin{align*}
 T_{3}x_{k} = \big(2- \tfrac{1}{k}, \tfrac{1}{4k}\big).
 \end{align*}
 Hence
 \begin{align*}
 \CC{\mathcal{S}}x_{k} = CC \big\{(-2,0), (2,0),\big(2- \tfrac{1}{k}, \tfrac{1}{4k}\big) \big\}.
 \end{align*}
By \cref{exam:ProperCirMapDiscontinuous}\cref{exam:ProperCirMapDiscontinuous:ii}, we obtain that
\begin{align*}
\lim_{k \rightarrow \infty} \CC{\mathcal{S}}x_{k} =(0,-8)=\overline{x} =\lim_{k \rightarrow \infty} x_{k}.
\end{align*}

 \cref{exam:DemicloseTRUE3:overlinex}: By \cref{exam:DemicloseTRUE3:fix}, $\overline{x}
 \not\in \Fix \CC{\mathcal{S}} $. 
 Therefore, the demiclosedness principle does not hold 
for  $\CC{\mathcal{S}}$.
\end{proof}

\begin{figure}[H] 
\begin{center} \includegraphics[scale=0.3]{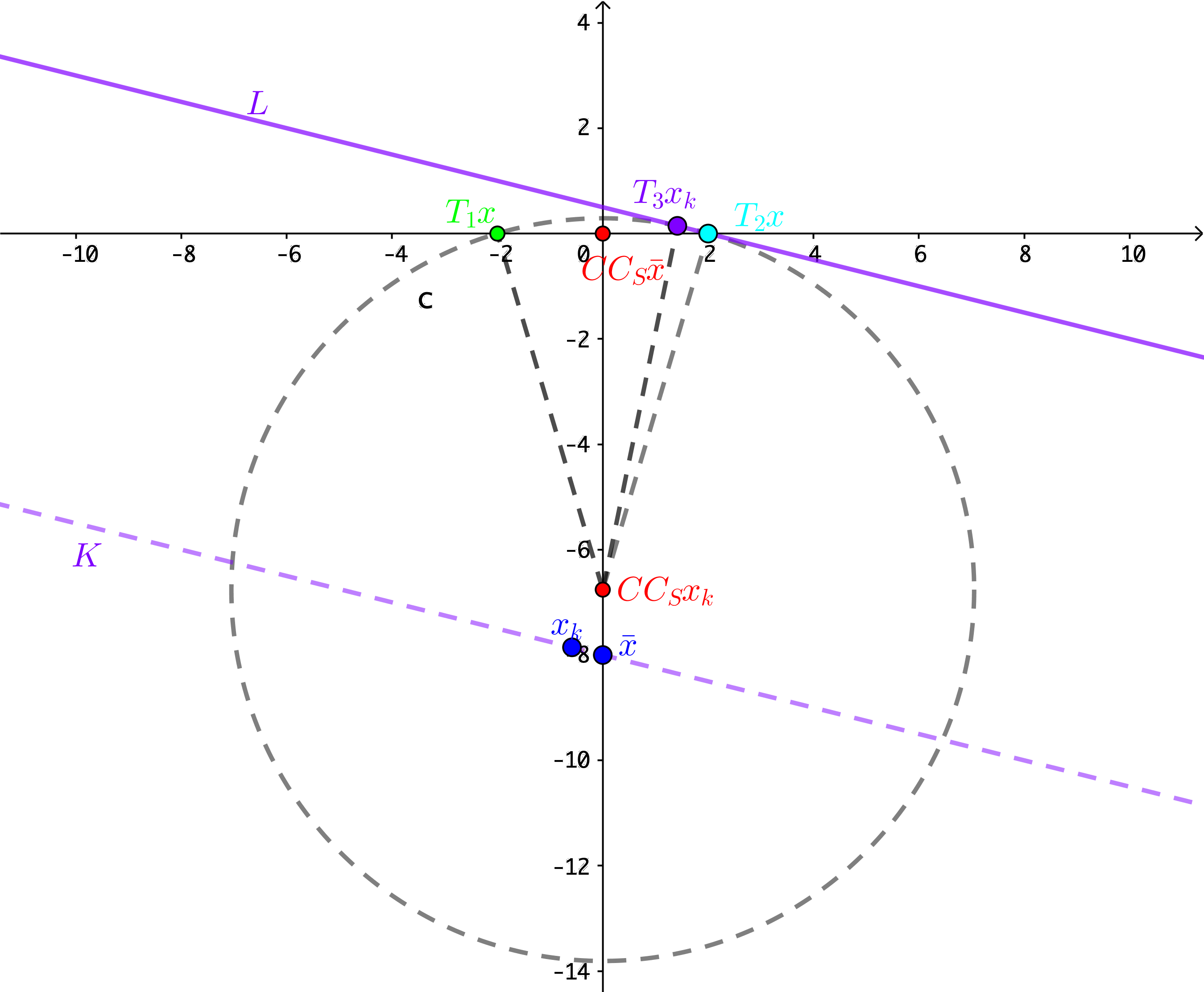}
\end{center}
\caption{\cref{exam:DemicloseTRUE3} illustrates that the demiclosedness principle may not hold for $\CC{\mathcal{S}}$.} \label{Demiconti.png}
\end{figure}

\begin{remark} \label{rem:Demiclosedness:holds:discon}
Consider \cref{exam:DemicloseTRUE3} where 
each $T_i$ is a projector and thus firmly nonexpansive
but $\Fix \CC{\mathcal{S}} = \varnothing$.
Is it possible to obtain an example where
the demiclosedness principle does not hold but 
yet $\Fix \CC{\mathcal{S}} \neq \varnothing$?
We do not know the answer to this question. 
\end{remark}

%---------------------------------------------Circumcenter mapping induced by reflectors---------------------------------------------
%-----------------------------------------------------------------------------------------------------------------------------------------------------
%----------------------------------------------------------------------------------------------------------------------------------------------------
\section{Circumcenter mappings induced by reflectors} \label{sec:CircumMappingReflectors}
Recall that $m \in \mathbb{N} \smallsetminus \{0\}$ and that $U_{1}, \ldots, U_{m}$ are closed affine subspaces in the real Hilbert space $\mathcal{H}$ with $\cap^{m}_{i=1} U_{i} \neq \varnothing$. In the whole section, denote 
\begin{empheq}[box = \mybluebox]{equation} \label{eq:DefinitionOmega}
\Omega = \Big\{ \R_{U_{i_{r}}}\cdots \R_{U_{i_{2}}}\R_{U_{i_{1}}}  ~\Big|~ r \in \mathbb{N}, ~\mbox{and}~ i_{1}, i_{2},\ldots,  i_{r} \in \{1, \ldots,m\}    \Big\}.
\end{empheq}
By the empty product convention, $\prod^{0}_{j=1}\R_{U_{i_{j}}} =\Id$. So, when $r=0$ in \cref{eq:DefinitionOmega}, $\Id \in \Omega$. Hence, $\Omega$ is the set consisting of the identity operator, $\Id$, and all of the compositions of $(\forall i \in \{1, \ldots,m\})$ $\R_{U_{i}}$.

Throughout this section, we assume that 
\begin{empheq}[box = \mybluebox]{equation*}
\Id \in \mathcal{S} \subseteq \aff \Omega.
\end{empheq}

\subsection{Proper circumcenter mappings induced by reflectors}
Note that for every $T$ in $\mathcal{S}$, where $ \mathcal{S} \subseteq \Omega$, there exists $r \in \mathbb{N}$ and $i_{1}, \ldots, i_{r} \in \{1, \ldots, m\}$ such that $T=\R_{U_{i_{r}}}\cdots \R_{U_{i_{2}}}\R_{U_{i_{1}}} $. Therefore, from now on we assume
\begin{empheq}[box = \mybluebox]{equation*} 
\R_{U_{i_{r}}}\cdots \R_{U_{i_{2}}}\R_{U_{i_{1}}}~\text{is a representative element of the set}~\mathcal{S},~\text{where}~ \Id \in \mathcal{S} \subseteq \Omega.
\end{empheq}

We start with a useful lemma. 
\begin{lemma} \label{lem:EquDistS}
Assume that  $\Id \in \mathcal{S} \subseteq \Omega$. Let $x \in \mathcal{H}$. Then for every $\R_{U_{i_{r}}}\cdots \R_{U_{i_{2}}}\R_{U_{i_{1}}}  \in \mathcal{S}$,
\begin{align*}
(\forall u \in \cap^{m}_{i=1} U_{i}) \quad \norm{x- u}=\norm{\R_{U_{i_{r}}}\cdots \R_{U_{i_{2}}}\R_{U_{i_{1}}}x-u}.
\end{align*}
\end{lemma}

\begin{proof}
Let $u \in \cap^{m}_{i=1} U_{i}$. Because $U_{1}, \ldots ,U_{m}$ are closed affine subspaces and $u \in \cap^{m}_{i=1} U_{i} \subseteq \cap^{r}_{j=1} U_{i_{j}} $, by \cref{prop:PR}\cref{prop:PR:ReflectorEquidist}, we have
\begin{align*}
 \norm{x-u} &= \norm{\R_{U_{i_{1}}}x-u} ~~~~~~(\text{by}~ u \in U_{i_{1}})\\
 \norm{\R_{U_{i_{1}}}x-u} &= \norm{\R_{U_{i_{2}}}\R_{U_{i_{1}}}x-u} ~~~~~~(\text{by}~ u \in U_{i_{2}})\\
&\cdots \\
\norm{\R_{U_{i_{r-1}}}\cdots \R_{U_{i_{2}}}\R_{U_{i_{1}}}x-u} &= \norm{\R_{U_{i_{r}}} \R_{U_{i_{r-1}}}\cdots \R_{U_{i_{2}}}\R_{U_{i_{1}}}x-u} ~~~~~~(\text{by}~ u \in U_{i_{r}}),
\end{align*} 
which yield
\begin{align*}
\norm{x-u}=\norm{\R_{U_{i_{r}}}\R_{U_{i_{r-1}}}\cdots \R_{U_{i_{2}}}\R_{U_{i_{1}}}x-u}.
\end{align*}
\end{proof}

\begin{proposition} \label{thm:CCS:proper:Lem}
Assume that $\Id \in \mathcal{S} \subseteq \Omega$. Let $x \in \mathcal{H}$. Then for every $u \in \cap^{m}_{i=1} U_{i}$,
\begin{enumerate}
\item \label{thm:CCS:proper:belong} $\Pro_{\aff (\mathcal{S}(x))}(u) \in \aff (\mathcal{S}(x))$, and 
\item  \label{thm:CCS:proper:EquaDistance} for every $\R_{U_{i_{r}}}\cdots \R_{U_{i_{2}}}\R_{U_{i_{1}}} \in \mathcal{S}$,
\begin{align*}
\norm{\Pro_{\aff (\mathcal{S}(x))}(u) -x } =\norm{\Pro_{\aff (\mathcal{S}(x))}(u) -\R_{U_{i_{r}}}\cdots \R_{U_{i_{2}}}\R_{U_{i_{1}}}x}.
\end{align*}
\end{enumerate}
\end{proposition}

\begin{proof}
\cref{thm:CCS:proper:belong}: Let $u \in \cap^{m}_{i=1} U_{i}$. Because $\aff (\mathcal{S}(x))$ is the translate of a finite-dimensional linear subspace, $\aff (\mathcal{S}(x))$ is a closed affine subspace. Hence, we know $\Pro_{\aff (\mathcal{S}(x))}(u)$ is well-defined. 
Clearly, $\Pro_{\aff (\mathcal{S}(x))}(u) \in \aff (\mathcal{S}(x))$, i.e., \cref{thm:CCS:proper:belong} is true. 

\cref{thm:CCS:proper:EquaDistance}: Take an arbitrary but fixed element $\R_{U_{i_{r}}}\cdots \R_{U_{i_{2}}}\R_{U_{i_{1}}}$ in $\mathcal{S}$. Since $\Id, \R_{U_{i_{r}}}\cdots \R_{U_{i_{2}}}\R_{U_{i_{1}}} \in \mathcal{S}$, we know $x, \R_{U_{i_{r}}}\cdots \R_{U_{i_{2}}}\R_{U_{i_{1}}}x \in \mathcal{S}(x) \subseteq \aff (\mathcal{S}(x))$. Denote $p = \Pro_{\aff (\mathcal{S}(x))}(u)$. Substituting $C=\aff (\mathcal{S}(x))$, $x=u$ and $v=x$ in \cref{prop:PR}\cref{prop:PR:Pythagoras}, we deduce
\begin{align}\label{thm:eq:1}
\norm{u-p}^{2}+\norm{x-p}^{2} &= \norm{u-x}^{2}. 
\end{align}
Similarly, substitute $C=\aff (\mathcal{S}(x))$, $x=u$ and $v=\R_{U_{i_{r}}}\cdots \R_{U_{i_{2}}}\R_{U_{i_{1}}}x$ in \cref{prop:PR}\cref{prop:PR:Pythagoras} to obtain
\begin{align} \label{thm:eq:2}
\norm{u-p}^{2}+\norm{\R_{U_{i_{r}}}\cdots \R_{U_{i_{2}}}\R_{U_{i_{1}}}x-p}^{2} &= \norm{u-\R_{U_{i_{r}}}\cdots \R_{U_{i_{2}}}\R_{U_{i_{1}}}x}^{2}.  
\end{align}
On the other hand, by \cref{lem:EquDistS}, we know
\begin{align}  \label{thm:eq:RRxu}
\norm{x-u}=\norm{\R_{U_{i_{r}}}\cdots \R_{U_{i_{2}}}\R_{U_{i_{1}}}x-u}.
\end{align}
Combining \cref{thm:eq:RRxu} with \cref{thm:eq:1} and \cref{thm:eq:2}, we yield 
\begin{align*}
\norm{p-x} =\norm{p-\R_{U_{i_{r}}}\cdots \R_{U_{i_{2}}}\R_{U_{i_{1}}}x}.
\end{align*}
Since $\R_{U_{i_{r}}}\cdots \R_{U_{i_{2}}}\R_{U_{i_{1}}} \in \mathcal{S}$ is arbitrary, thus \cref{thm:CCS:proper:EquaDistance} holds.
\end{proof}

Combining \cref{thm:CW:ExistWellDefined} with \cref{thm:CCS:proper:Lem}, we deduce the theorem below which is one of the main results in this paper. 

\begin{theorem} \label{thm:CCS:proper}
Assume that $\Id \in \mathcal{S} \subseteq \Omega$. 
Then the following hold:
\begin{enumerate}
\item \label{thm:CCS:proper:prop} The circumcenter mapping $\CC{\mathcal{S}} : \mathcal{H} \rightarrow \mathcal{H}$ induced by $\mathcal{S}$ is $\emph{proper}$, i.e., for every $x \in \mathcal{H}$, $\CC{\mathcal{S}}x$ is the unique point satisfying the two conditions below: 
\begin{enumerate}
\item  \label{set:CW:i} $\CC{\mathcal{S}}x\in  \aff (\mathcal{S}(x))$, and 
\item  \label{set:CW:ii} $(\forall \R_{U_{i_{k}}}\cdots \R_{U_{i_{1}}} \in \mathcal{S})$ $\norm{\CC{\mathcal{S}}x-x} =\norm{\CC{\mathcal{S}}x-\R_{U_{i_{k}}}\cdots \R_{U_{i_{1}}}x}$.
\end{enumerate}
\item \label{thm:CCS:AllU} $(\forall x \in \mathcal{H})$  $(\forall u \in \cap^{m}_{i=1} U_{i})$  $\CC{\mathcal{S}}x= \Pro_{\aff (\mathcal{S}(x))}(u)$.

\item \label{thm:CCS:PaffU} $(\forall x \in \mathcal{H})$ $\CC{\mathcal{S}}x= \Pro_{\aff (\mathcal{S}(x))}(\Pro_{\cap^{m}_{i=1} U_{i}} x)$.
\end{enumerate}
\end{theorem}

\begin{proof}
\cref{thm:CCS:proper:prop} and \cref{thm:CCS:AllU}: The required results follow from \cref{thm:CW:ExistWellDefined} and \cref{thm:CCS:proper:Lem}. 

\cref{thm:CCS:PaffU}: Since $\Pro_{\cap^{m}_{i=1} U_{i}} x \in \cap^{m}_{i=1} U_{i}$, the desired result comes from \cref{thm:CCS:AllU}. 
\end{proof}

We now list several proper circumcenter mappings induced by reflectors; the
properness of some of these mappings is derived from \cref{thm:CCS:proper}.

\begin{example} \label{ex:m:m} Assume that $\mathcal{S}=\{\Id, \R_{U_{1}},
\R_{U_{2}}, \ldots, \R_{U_{m}}\}$. By 
\cref{thm:CCS:proper}\cref{thm:CCS:proper:prop}, $\CC{\mathcal{S}}$ is
proper.
\end{example}

\begin{example} \label{ex:m:mm} Assume that $\mathcal{S}=\{\Id,
\R_{U_{2}}\R_{U_{1}}, \R_{U_{3}}\R_{U_{2}}, \ldots, \R_{U_{m}}\R_{U_{m-1}},
\R_{U_{1}}\R_{U_{m}}\}$. By
\cref{thm:CCS:proper}\cref{thm:CCS:proper:prop}, $\CC{\mathcal{S}}$ is
proper.
\end{example}

\begin{example}[Behling et al.\ \cite{BCS2017}]
\label{ex:2:3} Assume that $m=2$ and that $\mathcal{S}=\{\Id, \R_{U_{1}},
\R_{U_{2}}\R_{U_{1}} \}$. Then, by
\cref{thm:CCS:proper}\cref{thm:CCS:proper:prop}, $\CC{\mathcal{S}}$ is
proper.
\end{example}

\begin{example}[Behling et al.\ \cite{BCS2018}]
\label{ex:m:mm:R} Assume that $\mathcal{S}=\{\Id, \R_{U_{1}},
\R_{U_{2}}\R_{U_{1}}, \ldots, \R_{U_{m}}\cdots \R_{U_{2}}\R_{U_{1}}\}$. Then,
by \cref{thm:CCS:proper}\cref{thm:CCS:proper:prop}, $\CC{\mathcal{S}}$ is
proper.
\end{example}

\begin{remark}
In fact, the C--DRM operator $C_{T}$ defined in {\rm
\cite[Section~2]{BCS2017}} is the $\CC{\mathcal{S}}$ operator of 
\cref{ex:2:3} while the CRM operator $C$ defined in {\rm
\cite[page~159]{BCS2018}} is the
operator $\CC{\mathcal{S}}$ from \cref{ex:m:mm:R}.
\end{remark}

\begin{example} \label{ex:2:2} Assume that $m=2$ and that $\mathcal{S}=\{\Id,
\R_{U_{2}}\R_{U_{1}} \}$. By \cref{prop:form:m2:Oper},
\begin{align*}
\CC{\mathcal{S}} =\frac{\Id+\R_{U_{2}}\R_{U_{1}}}{2},
\end{align*}
which is the well-known Douglas--Rachford splitting operator. 
Clearly, $\CC{\mathcal{S}}$ is proper.
\end{example}

\begin{example} \label{ex:m:2:Proj} Assume that $m=2$ and that $\mathcal{S}=\{\Id,
\R_{U_{1}}, \R_{U_{2}}\}$. Then $\CC{\mathcal{S}}$ is proper. Moreover,
\begin{align*}
(\forall x \in U_{1}) \quad \CC{\mathcal{S}}x=\Pro_{U_{2}}x \quad \text{and} \quad (\forall x \in U_{2}) \quad \CC{\mathcal{S}}x=\Pro_{U_{1}}x.
\end{align*}
\end{example}
\begin{proof}
The first assertion follows from \cref{ex:m:m}. 
As for the remaining ones, note that
\begin{align} \label{eq:ex:m:2:Proj}
(\forall x \in U_{1}) \quad \mathcal{S}(x)=\{x, \R_{U_{2}}x\} \quad \text{and} \quad (\forall x \in U_{2}) \quad \mathcal{S}(x)=\{x, \R_{U_{1}}x\}.
\end{align}
Combining \cref{eq:ex:m:2:Proj} with \cref{prop:form:m2:Oper}, we obtain that
\begin{align*}
(\forall x \in U_{1}) \quad \CC{\mathcal{S}}x=\frac{x+\R_{U_{2}}x}{2}=\Pro_{U_{2}}x \quad \text{and} \quad (\forall x \in U_{2}) \quad  \CC{\mathcal{S}}x=\frac{x+\R_{U_{1}}x}{2}=\Pro_{U_{1}}x.
\end{align*}
The proof is complete.
\end{proof}

\begin{example} \label{exam:OurCCS}
Assume that $m=2$ and that 
$\mathcal{S}=\{\Id, \R_{U_{1}}, \R_{U_{2}}, \R_{U_{2}}\R_{U_{1}} \}$. 
Let $x \in \mathcal{H}$ and set 
$l = \card \{ x, \R_{U_{1}}x,\R_{U_{2}}x, \R_{U_{2}}\R_{U_{1}}x \}$. Then exactly one of the following cases occurs.
\begin{enumerate}
\item \label{thm:SymForm:1} $l=1$ and $\CC{\mathcal{S}}x =x$.
\item \label{thm:SymForm:2} $l=2$, say $S(x)=\{x_{1},x_{2}\}$, 
where $x_{1}$ and $x_{2}$ are two distinct elements in $S(x)$, and $\CC{\mathcal{S}}x=\frac{x_{1}+x_{2}}{2}$.
\item \label{thm:SymForm:3} 
$l=3$, say $S(x)=\{x_{1},x_{2},x_{3}\}$, 
where $x_{1}$, $x_{2}$, $x_{3}$ are pairwise distinct elements in $S(x)$, and
\begin{align*}
\scalemath{0.9}{\CC{\mathcal{S}}x = \frac{\norm{x_{2}-x_{3}}^{2} \innp{x_{1}-x_{3},x_{1}-x_{2}}x_{1}+ \norm{x_{1}-x_{3}}^{2} \innp{x_{2}-x_{3},x_{2}-x_{1}}x_{2}+ \norm{x_{1}-x_{2}}^{2} \innp{x_{3}-x_{1},x_{3}-x_{2}}x_{3} }{2(\norm{x_{2}-x_{1}}^{2}\norm{x_{3}-x_{1}}^{2}- \innp{x_{2}-x_{1}, x_{3}-x_{1}}^{2})}}.
\end{align*}
\item \label{thm:SymForm:4} $l=4$ and
\begin{align*}
\CC{\mathcal{S}}x= x_{1}+\frac{1}{2}(x_{2}-x_{1},\ldots,x_{t_{x}}-x_{1})
 G( x_{2}-x_{1},\ldots,x_{t_{x}}-x_{1})^{-1}
\begin{pmatrix}
 \norm{x_{2}-x_{1}}^{2} \\
 \vdots\\
\norm{x_{t_{x}}-x_{1}}^{2} \\
\end{pmatrix},
\end{align*}
where $\{x_{1},x_{2},\ldots,x_{t_{x}}\}=S(x)$, and $x_{1},x_{2},\ldots,x_{t_{x}}$ are affinely independent and 
\begin{align*}
&\quad \quad G( x_{2}-x_{1},\ldots, x_{t_{x}-1}-x_{1}, x_{t_{x}}-x_{1})\\
&=
\begin{pmatrix} 
\norm{x_{2}-x_{1}}^{2} &\innp{x_{2}-x_{1},x_{3}-x_{1}} & \cdots & \innp{x_{2}-x_{1}, x_{t_{x}}-x_{1}}  \\ 
\vdots & \vdots & ~~& \vdots \\
\innp{x_{t_{x}-1}-x_{1},x_{2}-x_{1}} & \innp{x_{t_{x}-1}-x_{1}, x_{3}-x_{1}} & \cdots & \innp{x_{t_{x}-1}-x_{1},x_{t_{x}}-x_{1}} \\
\innp{x_{t_{x}}-x_{1},x_{2}-x_{1}} & \innp{x_{t_{x}}-x_{1},x_{3}-x_{1}} & \cdots & \norm{x_{t_{x}}-x_{1}}^{2} \\
\end{pmatrix}.
\end{align*} 
\end{enumerate}
\end{example}

\begin{proof}
By \cref{thm:CCS:proper}\cref{thm:CCS:proper:prop}, $\CC{\mathcal{S}}$ is
proper. The rest follows from
\cref{fact:CircForTwoPoints} and \cref{fact:unique:LinIndpPformula}.
\end{proof}

We now turn to the properness of 
 $\CC{\widetilde{\mathcal{S}}}$ 
 when $\Id \in \widetilde{\mathcal{S}} \subseteq \aff \Omega$.

\begin{proposition} \label{prop:proper:affRU}
Let $\alpha \in \mathbb{R}$. Assume that 
\begin{align} \label{eq:prop:proper:affRU:assu}
\widetilde{\mathcal{S}} = \{\Id, (1-\alpha) \Id + \alpha \R_{U_{1}}, \ldots, (1-\alpha) \Id + \alpha \R_{U_{m}} \},
\end{align}
and  that 
\begin{align} \label{eq:prop:proper:affRU:S}
\mathcal{S} = \{\Id, \R_{U_{1}}, \ldots,  \R_{U_{m}}\}.
\end{align}
Then $\CC{\widetilde{\mathcal{S}}}$ is proper. Moreover,
\begin{align}  \label{eq:prop:proper:affRU:resul}
(\forall x \in \mathcal{H}) \quad  \CC{\widetilde{\mathcal{S}}}x =\alpha  \CC{\mathcal{S}}x  + (1 -  \alpha )x \in \mathcal{H}.
\end{align}
\end{proposition}
\begin{proof}
If $\alpha=0$, then $\widetilde{\mathcal{S}} = \{\Id \}$, by \cref{def:cir:map},
\begin{align*}
(\forall x \in \mathcal{H}) \quad \CC{\widetilde{\mathcal{S}}}x =x =0 \CC{\mathcal{S}}x +(1-0)x \in \mathcal{H}.
\end{align*}
Now assume $\alpha \neq 0$. Let $x \in \mathcal{H}$. For every $i \in \{1, \ldots, m\}$, thus
\begin{align*}
\CC{\widetilde{\mathcal{S}}}x  & = \CCO ( \widetilde{\mathcal{S} }(x)) \quad (\text{by \cref{def:cir:map}})\\
& = \CCO { \Big ( \big \{ x, (1-\alpha)x + \alpha \R_{U_{1}} x,  \ldots, (1-\alpha) x + \alpha \R_{U_{m}} x \big\} \Big )} \quad (\text{by \cref{eq:prop:proper:affRU:assu}})\\
& = \CCO {\Big ( \big \{ 0, \alpha (\R_{U_{1}} x-x),  \ldots,  \alpha (\R_{U_{m}} x-x) \big\} +x \Big )} \\
& =  \CCO {\Big ( \big \{ 0, \alpha (\R_{U_{1}} x-x),  \ldots,  \alpha (\R_{U_{m}} x-x) \big\}  \Big )}+x \quad (\text{by \cref{fact:CircumSubaddi}})\\
& = \alpha \CCO {\Big ( \big \{ 0,  \R_{U_{1}} x-x,  \ldots,  \R_{U_{m}} x-x \big\}  \Big )}+x \quad (\text{by \cref{fact:CircumHomoge} and } \alpha \neq 0)\\
& = \alpha \CCO {\Big ( \big \{ x,  \R_{U_{1}} x,  \ldots,  \R_{U_{m}} x \big\} -x \Big )}+x\\
& = \alpha \CCO {\Big ( \big \{ x,  \R_{U_{1}} x,  \ldots,  \R_{U_{m}} x \big\}  \Big )}- \alpha x +x  \quad (\text{by \cref{fact:CircumSubaddi}}) \\
& = \alpha  \CCO {\big(\mathcal{S}(x) \big)} +(1- \alpha)x  \quad (\text{by \cref{eq:prop:proper:affRU:S} })  \\
& = \alpha  \CC{\mathcal{S}}x  + (1 -  \alpha )x \in \mathcal{H}. \quad (\text{by \cref{def:cir:map} and   \cref{thm:CCS:proper}\cref{thm:CCS:proper:prop}})
\end{align*}
The proof is complete.
\end{proof}

\begin{proposition} \label{prop:TT:CWhat}
Assume that $\mathcal{S}= \{ \Id , \R_{U_{2}}\R_{U_{1}},
\R_{U_{2}}\R_{U_{1}}\R_{U_{2}}\R_{U_{1}} \}$, set $T=\frac{\Id +
\R_{U_{2}}\R_{U_{1}}}{2}$, which is 
the Douglas--Rachford splitting operator,
and set $\widetilde{\mathcal{S}} =\{\Id, T, T^{2}\}$.
Then the
following hold:
\begin{enumerate}
\item \label{prop:TT:CWhat:AffS} $\aff \{\Id, T, T^{2}\} =\aff \mathcal{S}$.
\item \label{prop:TT:CWhat:Prop} $\CC{\widetilde{\mathcal{S}}} $ is proper.
\end{enumerate}
\end{proposition}

\begin{proof}
\cref{prop:TT:CWhat:AffS}: By \cref{fact:AffineHull}, 
\begin{align}
& \aff \{ \Id, T, T^{2}\}  =\aff (\mathcal{S}) \nonumber \\
\Longleftrightarrow & \Id + \spn\{T-\Id, T^{2}-\Id \} =\Id + \spn\{\R_{U_{2}}\R_{U_{1}}- \Id, \R_{U_{2}}\R_{U_{1}}\R_{U_{2}}\R_{U_{1}} - \Id\}. \label{eq:prop:TT:CWhat:affspn}
\end{align}
On the other hand,
\begin{align} \label{eq:TT:CWhat:1}
T-\Id = \frac{\R_{U_{2}}\R_{U_{1}}+\Id}{2}-\Id =\frac{\R_{U_{2}}\R_{U_{1}}-\Id}{2},
\end{align}
and
\begin{align} \label{eq:TT:CWhat:2}
T^{2} - \Id & = T^{2}- T +T -\Id= (T-\Id)T + (T-\Id) \nonumber \\
& = \frac{\R_{U_{2}}\R_{U_{1}}-\Id}{2} \Big(\frac{\R_{U_{2}}\R_{U_{1}}+\Id}{2} \Big)+\frac{\R_{U_{2}}\R_{U_{1}} - \Id}{2}\nonumber  \\
& = \frac{1}{4}(\R_{U_{2}}\R_{U_{1}}\R_{U_{2}}\R_{U_{1}}- \Id)+\frac{1}{2}(\R_{U_{2}}\R_{U_{1}}- \Id),
\end{align}
which result in 
\begin{align} \label{eq:TwoVectNonsingMatrix}
\begin{pmatrix} T - \Id &T^{2}-\Id \end{pmatrix} = \begin{pmatrix} \R_{U_{2}}\R_{U_{1}} - \Id &\R_{U_{2}}\R_{U_{1}}\R_{U_{2}}\R_{U_{1}}-\Id \end{pmatrix}\begin{pmatrix} \frac{1}{2} &\frac{1}{2} \\ 0 & \frac{1}{4} \end{pmatrix}.
\end{align}
Set $A= \begin{pmatrix} \frac{1}{2} &\frac{1}{2} \\ 0 & \frac{1}{4}
\end{pmatrix}$. Since $\det(A) =\frac{1}{8} \neq 0$,
\cref{eq:TwoVectNonsingMatrix} yields 
\begin{align} \label{eq:prop:TT:CWhat:spn}
\spn\{T-\Id, T^{2}-\Id \} = \spn\{\R_{U_{2}}\R_{U_{1}}- \Id, \R_{U_{2}}\R_{U_{1}}\R_{U_{2}}\R_{U_{1}} - \Id\}.
\end{align}
Altogether, \cref{eq:prop:TT:CWhat:spn} and \cref{eq:prop:TT:CWhat:affspn}
demonstrate to us that \cref{prop:TT:CWhat:AffS} is true.

\cref{prop:TT:CWhat:Prop}:
If $x, Tx, T^{2}x$ are affinely independent, by \cref{fact:unique:LinIndpPformula}, then $\CC{\widetilde{\mathcal{S}}}x \in \mathcal{H}$. Suppose $x, Tx, T^{2}x$ are affinely dependent. By \cref{eq:TwoVectNonsingMatrix} and $\det (A) \neq 0$, in this case, $x, \R_{U_{2}}\R_{U_{1}}x,  \R_{U_{2}}\R_{U_{1}}\R_{U_{2}}\R_{U_{1}}x$ are affinely dependent. Applying \cref{thm:CCS:proper}\cref{thm:CCS:proper:prop}, we know $\CC{\mathcal{S}}x \in \mathcal{H}$. Hence, \cref{fact:clform:three} yields that
\begin{align} \label{eq:prop:TT:CWhat:cardSx}
\card \Big(\{ x , \R_{U_{2}}\R_{U_{1}},x \R_{U_{2}}\R_{U_{1}}\R_{U_{2}}\R_{U_{1}}x \}  \Big) = \card \big(\mathcal{S}(x) \big) \leq 2. 
\end{align}
If $Tx-x=0$, by \cref{prop:form:m2:Oper}, $\CC{\widetilde{\mathcal{S}}} x=\frac{x+T^{2}x}{2}$. Now suppose $Tx-x \neq 0$. By \cref{eq:TT:CWhat:1}, $\R_{U_{2}}\R_{U_{1}}x \neq x$. Therefore, by \cref{eq:prop:TT:CWhat:cardSx} and $\R_{U_{2}}\R_{U_{1}}x \neq x$, either $\R_{U_{2}}\R_{U_{1}}\R_{U_{2}}\R_{U_{1}}x=\R_{U_{2}}\R_{U_{1}}x$ or $\R_{U_{2}}\R_{U_{1}}\R_{U_{2}}\R_{U_{1}}x=x$. Suppose $\R_{U_{2}}\R_{U_{1}}\R_{U_{2}}\R_{U_{1}}x=\R_{U_{2}}\R_{U_{1}}x$. Multiply both sides by $\R_{U_{1}}\R_{U_{2}}$, by \cref{lem:PU:RUIdempotent}\cref{lem:RUIdempotent}, to deduce
$\R_{U_{2}}\R_{U_{1}}x=x$,
which contradicts with $\R_{U_{2}}\R_{U_{1}}x$ $\neq$ $x$. 
Suppose $\R_{U_{2}}\R_{U_{1}}\R_{U_{2}}\R_{U_{1}}x$ $=$ $x$, by \cref{eq:TT:CWhat:1} and \cref{eq:TT:CWhat:2}, which implies,  $Tx=T^{2}x$.
Then by \cref{prop:form:m2:Oper}, we obtain $ \CC{\widetilde{\mathcal{S}}}x=\frac{x+Tx}{2} \in \mathcal{H}$.

In conclusion, $(\forall x \in \mathcal{H})$ $ \CC{\widetilde{\mathcal{S}}}x \in \mathcal{H}$, which means \cref{prop:TT:CWhat:Prop} holds.
\end{proof}

%%%%%-----------------------Improper circumcenter mappings induced by reflectors-----------------------%%%%%

\subsection{Improper circumcenter mappings induced by reflectors} \label{subsec:ImpcircuMappingRefle}
Propositions \ref{prop:proper:affRU} and \ref{prop:TT:CWhat} naturally prompt
the following question: Is $\CC{\widetilde{\mathcal{S}}}$ proper for every
$\widetilde{S}$ with $\Id \in \widetilde{S} \subseteq \aff \Omega$ ?
The following examples provide negative answers.

\begin{example} \label{exam:IMp:affRU}
Assume that $m=2$, that $ U:=U_{1} = U_{2} \subsetneqq \mathcal{H}$,
 and that $\{ \alpha_{1}, \alpha_{2} \} \subseteq \mathbb{R}$. Assume further that $\widetilde{\mathcal{S}} =\{\Id, (1-\alpha_{1})\Id + \alpha_{1} \R_{U}, (1- \alpha_{2})\Id + \alpha_{2} \R_{U}\}$. Then $\CC{\widetilde{\mathcal{S}}}$ is improper if and only if $\alpha_{1} \neq 0, \alpha_{2} \neq 0$ and $\alpha_{2} \neq \alpha_{1}$. 
\end{example}

\begin{proof}
By \cref{prop:form:m2:Oper}, when $\alpha_{1} =0$ or $\alpha_{2} =0$, then
$\CC{\widetilde{\mathcal{S}}}$ is proper.

For every  $x \in \mathcal{H}$, if $\alpha_{1} \neq 0$,
\begin{align*}
\big( (1-\alpha_{2})x + \alpha_{2} \R_{U}x \big )-x = \alpha_{2} (\R_{U}x -x) =\frac{\alpha_{2}}{\alpha_{1}} \alpha_{1} (\R_{U}x -x) = \frac{\alpha_{2}}{\alpha_{1}}  \Big( \big( (1-\alpha_{1})x + \alpha_{1} \R_{U}x \big )-x \Big),
\end{align*}
which implies that, by \cref{fac:AffinIndeLineInd}, 
\begin{align} \label{eq:exam:IMp:affRU:affind}
x, (1-\alpha_{1})x + \alpha_{1} \R_{U}x, (1- \alpha_{2})x + \alpha_{2} \R_{U}x ~\text{ are affinely dependent}.
\end{align}

On the other hand, if $x \in \mathcal{H} \smallsetminus U$, then since $\alpha_{1} \neq 0$, $\alpha_{2} \neq 0$, $\alpha_{1} \neq \alpha_{2}$ and $\Fix \R_{U} =U$, we obtain that 
\begin{subequations}  \label{eq:exam:IMp:affRU:card}
\begin{align}
& \quad \quad \card \big \{x, (1-\alpha_{1})x + \alpha_{1} \R_{U}x, (1- \alpha_{2})x + \alpha_{2} \R_{U}x \big\} =3 \\
& \Longleftrightarrow \alpha_{1} \neq 0, \alpha_{2} \neq 0 ~~\text{and}~~\alpha_{2} \neq \alpha_{1}. 
\end{align} 
\end{subequations}
Combining \cref{cor:T1T2T3Dom} with \cref{eq:exam:IMp:affRU:affind} and \cref{eq:exam:IMp:affRU:card}, we deduce the required result. 
\end{proof}

\begin{example} \label{exam::IMp:affRURU}
Assume that $m=2$, that 
$ U:=U_{1} = U_{2} \subsetneqq \mathcal{H}$, 
and that $\{\alpha_{1}, \alpha_{2} \} \subseteq \mathbb{R}$. Assume further
that $\widetilde{\mathcal{S}} = \{ \Id, (1-\alpha_{1}) \Id +\alpha_{1}\R_{U},
\big( (1-\alpha_{2})\Id +\alpha_{2} \R_{U} \big) \circ \big((1-\alpha_{1})
\Id +\alpha_{1}\R_{U} \big)\}$. Then $\CC{\widetilde{\mathcal{S}}}$ is
improper if and only if $\alpha_{1} \neq 0, \alpha_{1} \neq \frac{1}{2},
\alpha_{2} \neq 0$ and $\alpha_{2} \neq \frac{\alpha_{1}}{2 \alpha_{1}-1}$.
\end{example}
\begin{proof}
By \cref{prop:form:m2:Oper}, when $\alpha_{1} =0$ or $\alpha_{2} =0$, then $\CC{\widetilde{\mathcal{S}}}$ is proper.

Note that 
\begin{subequations} \label{eq:exam::IMp:affRURU:thirditem}
\begin{align} 
& \quad~ \big( (1-\alpha_{2})\Id +\alpha_{2} \R_{U} \big) \circ \big((1-\alpha_{1}) \Id +\alpha_{1}\R_{U} \big) \\
& =(1- \alpha_{1}-\alpha_{2}+\alpha_{2}\alpha_{1} )\Id + (\alpha_{1} -\alpha_{2}\alpha_{1})\R_{U} +(\alpha_{2} -\alpha_{2}\alpha_{1})\R_{U} +\alpha_{2}\alpha_{1}\R_{U}\R_{U} \\
& =(1- \alpha_{1}-\alpha_{2}+2 \alpha_{2}\alpha_{1}  ) \Id + (\alpha_{1}+\alpha_{2}-2 \alpha_{2}\alpha_{1}) \R_{U} \quad (\text{by \cref{lem:PU:RUIdempotent}\cref{lem:RUIdempotent}})
\end{align}
\end{subequations}
For every $x \in \mathcal{H}$, if $\alpha_{1} \neq 0$,
\begin{align*}
& \quad~\Big( \big( (1-\alpha_{2})\Id +\alpha_{2} \R_{U} \big) \circ \big((1-\alpha_{1}) \Id +\alpha_{1}\R_{U} \big)x \Big) -x \\
& =   (\alpha_{1}+\alpha_{2}-2 \alpha_{2}\alpha_{1}) ( \R_{U}x-x ) \quad (\text{by \cref{eq:exam::IMp:affRURU:thirditem}}) \\
& =  \frac{\alpha_{1}+\alpha_{2}-2 \alpha_{2}\alpha_{1}}{ \alpha_{1}} \alpha_{1} ( \R_{U}x-x ) \\
& =  \frac{\alpha_{1}+\alpha_{2}-2 \alpha_{2}\alpha_{1}}{ \alpha_{1}} \Big( \big( (1-\alpha_{1}) x+\alpha_{1}\R_{U}x \big) - x \Big),
\end{align*}
which implies, by \cref{fac:AffinIndeLineInd}, that
\begin{align}  \label{eq:exam::IMp:affRURU:AffiInd}
x, (1-\alpha_{1}) x +\alpha_{1}\R_{U}x, \big( (1-\alpha_{2})\Id +\alpha_{2} \R_{U} \big) \circ \big((1-\alpha_{1}) \Id +\alpha_{1}\R_{U} \big)x \text{ are affinely dependent.}
\end{align}
On the other hand, assume now $x \in \mathcal{H} \smallsetminus U$. Then
\begin{subequations} \label{eq:exam::IMp:affRURU:card}
\begin{align} 
& \quad \quad \card \Big\{x, (1-\alpha_{1}) x +\alpha_{1}\R_{U}x, \big( (1-\alpha_{2})\Id +\alpha_{2} \R_{U} \big) \circ \big((1-\alpha_{1}) \Id +\alpha_{1}\R_{U} \big)x  \Big\} =3 \\
& \Longleftrightarrow  \alpha_{1} \neq 0, \alpha_{1} \neq \frac{1}{2}, \alpha_{2} \neq 0 ~~\text{and}~~\alpha_{2} \neq \frac{\alpha_{1}}{2 \alpha_{1}-1}.
\end{align}
\end{subequations}
Combining \cref{cor:T1T2T3Dom} with \cref{eq:exam::IMp:affRURU:AffiInd} and
\cref{eq:exam::IMp:affRURU:card}, we infer the desired result.
\end{proof}

The following example is a special case of \cref{exam::IMp:affRURU}. 
\begin{example} \label{exam:Count:prop:TT:CWhat}
Assume that $m=2$, that $U_{1} \varsubsetneqq U_{2} = \mathcal{H}$, 
and that $\{\alpha_{1}, \alpha_{2} \} \subseteq \mathbb{R}$.  
Assume further that $\widetilde{\mathcal{S}} = \{ \Id, (1-\alpha_{1}) \Id
+\alpha_{1}\R_{U_{2}}\R_{U_{1}}, \big( (1-\alpha_{2})\Id +\alpha_{2}
\R_{U_{2}}\R_{U_{1}} \big) \circ \big((1-\alpha_{1}) \Id
+\alpha_{1}\R_{U_{2}}\R_{U_{1}} \big)\}$. Then $\CC{\widetilde{\mathcal{S}}}$
is improper if and only if $\alpha_{1} \neq 0, \alpha_{1} \neq \frac{1}{2},
\alpha_{2} \neq 0$ and $\alpha_{2} \neq \frac{\alpha_{1}}{2 \alpha_{1}-1}$.
\end{example}

\begin{proof}
Since $\R_{U_{2}} = \R_{\mathcal{H}} =\Id$, we deduce that 
\begin{align*}
\widetilde{\mathcal{S}} = \{ \Id, (1-\alpha_{1}) \Id +\alpha_{1}\R_{U_{1}}, \big( (1-\alpha_{2})\Id +\alpha_{2} \R_{U_{1}} \big) \circ \big((1-\alpha_{1}) \Id +\alpha_{1}\R_{U_{1}} \big)\}.
\end{align*}
The desired result follows directly from \cref{exam::IMp:affRURU}.
\end{proof}
Notice that in  \cref{prop:TT:CWhat} we showed that for $\widetilde{\mathcal{S}} =\{\Id, T, T^{2}\} = \{\Id,\frac{\Id + \R_{U_{2}}\R_{U_{1}}}{2}, \frac{\Id + \R_{U_{2}}\R_{U_{1}}}{2} \circ \frac{\Id + \R_{U_{2}}\R_{U_{1}}}{2} \}$,  $\CC{\widetilde{\mathcal{S}}}$ is proper. The example above says that this result is not a conincidence.

%%%%%%%%%%%%%%%%%%%%Two special circumcenter mappings induced by reflectors%%%%%%%%%%%

\subsection{Particular circumcenter mappings in finite-dimensional spaces}
\subsubsection{Application to best approximation} 
Suppose that 
\begin{align*}
\mathcal{S}_{1}=\{\Id, \R_{U_{1}}, \R_{U_{2}}\} \quad \text{and} \quad \mathcal{S}_{2} = \{\Id, \R_{U_{1}}, \R_{U_{2}}\R_{U_{1}}\}.
\end{align*}
By \cref{ex:m:m} and \cref{ex:2:3}, we know $\CC{S_{1}}$ and $\CC{S_{2}}$ are
proper. Hence, for every $x \in \mathcal{H}$, we are able to generate
iterations $(\CC{S_{1}}^{k}x)_{k\in \mathbb{N}}$ and $(\CC{S_{2}}^{k}x)_{k\in
\mathbb{N}}$. In the following two examples, we choose two linear subspaces,
$U_{1}$ and $U_{2}$, in $\mathbb{R}^{3}$ and one point $x_{0} \in
\mathbb{R}^{3}$. Then we count the iteration numbers needed for the four
algorithms: the shadow sequence of the Douglas--Rachford method (DRM) 
(see, \cite{BCNPW2014} for details), the sequence generated by the method
of alternating projections (MAP), and the sequence generated by iterating $\CC{S_{1}}$
and $\CC{S_{2}}$ to find the best approximation point $\overline{x} = \Pro_{U_{1}
\cap U_{2}}x_{0}$.
\begin{example} \label{exam:3D:LinePlane}
Assume that $\mathcal{H} =\mathbb{R}^{3}$, that $U_{1}$ is the line passing 
through the points $(0,0,0)$ and $(1,0,0)$, and that $U_{2}$ is the plane $\{(x,y,z)
~|~ x+y+z=0\}$. Let $x_{0} = (0.5,0,0)$. As \cref{table:Lineandplane} shows, both
of the $\CC{S_{1}}$ and $\CC{S_{2}}$ are faster than DRM and MAP. 
(The results were obtained using \texttt{GeoGebra}.)
\begin{table}[H]
\centering
\begin{tabu} to 0.8\textwidth { m{7cm}  m{6cm} }
\hline
Algorithm & Iterations needed to find $\Pro_{U_{1} \cap U_{2}}x_{0}$ \\
\hline
Douglas--Rachford method  & 12\\
Method of alternating projections & 12 \\
Circumcenter method induced by $\mathcal{S}_{1}$ & 1 \\
Circumcenter method induced by $\mathcal{S}_{2}$ & 1 \\
\hline
\end{tabu}
\caption{Iterations needed for each algorithm. 
See \cref{exam:3D:LinePlane} for details.}
\label{table:Lineandplane}
\end{table}
\end{example}
\begin{figure}[H] 
\begin{center} \includegraphics[scale=1,trim={0 3cm 0 0},clip]{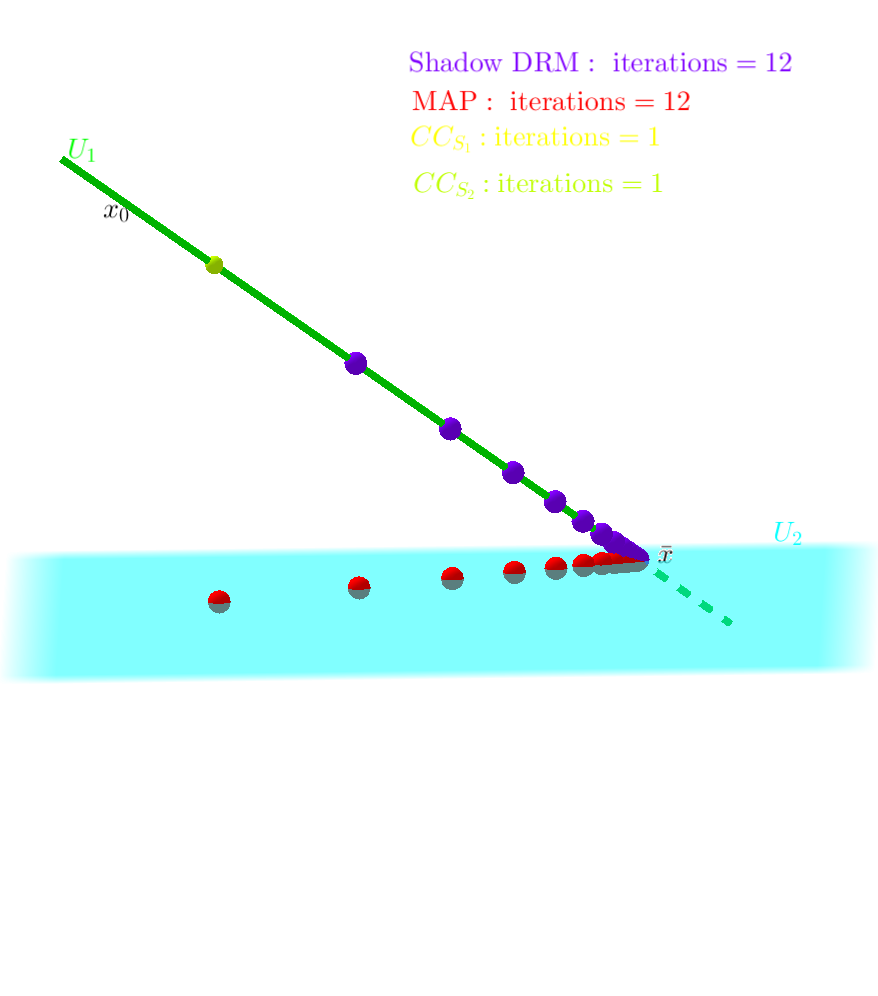}
\end{center}
\caption{\cref{exam:3D:LinePlane} compares iterations 
for a line and a plane.} \label{3DLINE_PLANE}
\end{figure}

\begin{example} \label{exam:3D:PlanePlane}
Assume that $\mathcal{H} =\mathbb{R}^{3}$, 
that $U_{1}=\{(x,y,z) ~|~ x+y+z=0\}$, and that $U_{2}:=\{(x,y,z) ~|~
-x+2y+2z=0\}$. Set $x_{0} = (-1,0.5,0.5)$. As \cref{table:planeandplane}
illustrates, $\CC{S_{2}}$ is faster than the other methods, and
$\CC{S_{1}}$ performs no worse than DRM or MAP.  
(The results were obtained using \texttt{GeoGebra}.)
\begin{table}[H]
\centering
\begin{tabu} to 0.8\textwidth { m{7cm}  m{6cm} }
\hline
Algorithm & Iterations needed to find $\Pro_{U_{1} \cap U_{2}}x_{0}$ \\
\hline
Douglas--Rachford method  & 5\\
Method of alternating projections & 6 \\
Circumcenter method induced by $\mathcal{S}_{1}$ & 5 \\
Circumcenter method induced by $\mathcal{S}_{2}$ & 2 \\
\hline
\end{tabu}
\caption{Iterations needed for each algorithm. 
See \cref{exam:3D:PlanePlane} for details.}
\label{table:planeandplane}
\end{table}
\end{example}
\begin{figure}[H] 
\begin{center} 
\includegraphics[scale=0.5,trim={0 3cm 0 0},clip]{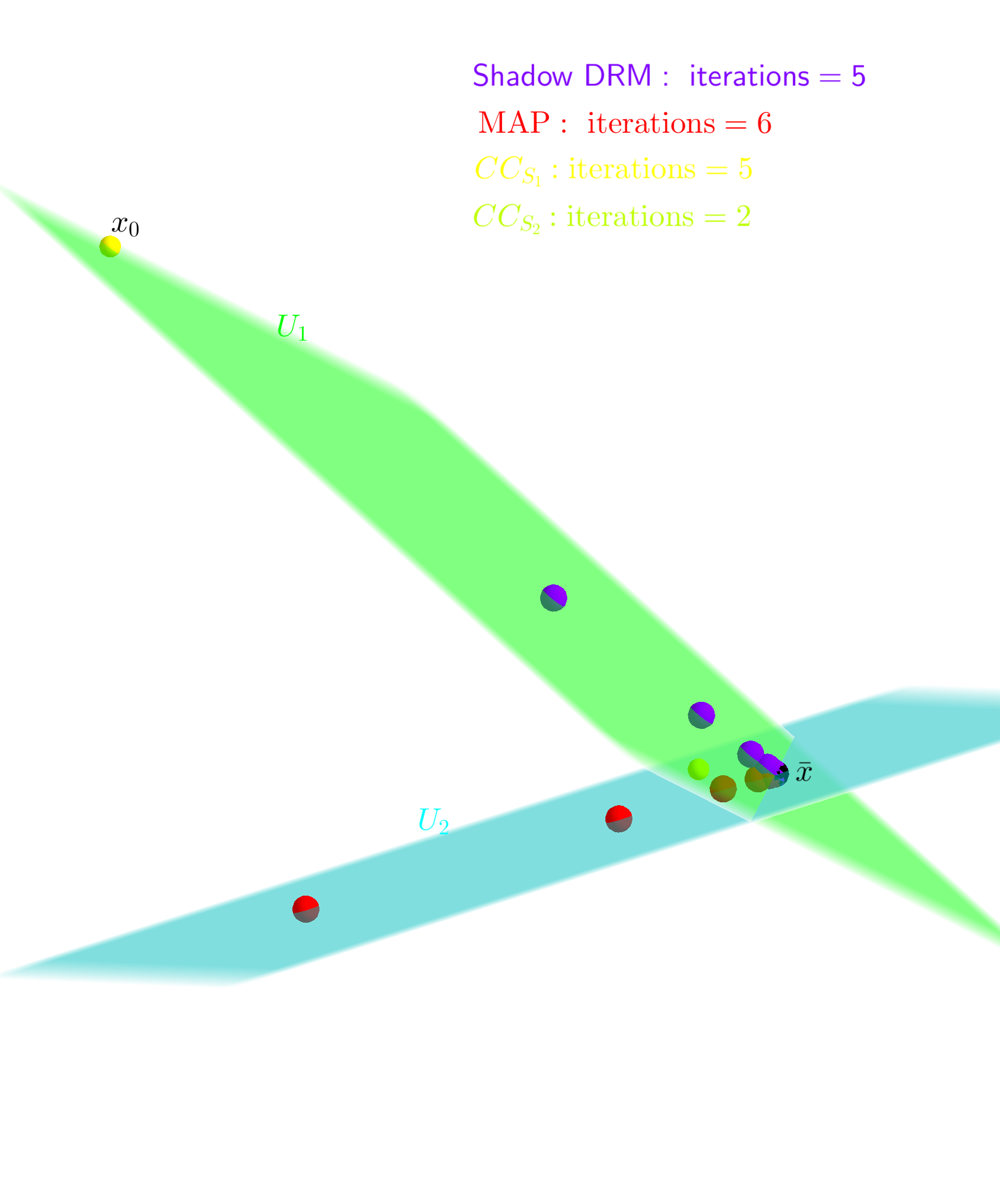}
\end{center}
\caption{\cref{exam:3D:PlanePlane} compares iterations for two planes.}
\label{3DPLANE_PLANE.png}
\end{figure}

\subsubsection{Counterexamples}

The following two examples show that the circumcenter mapping induced by
reflectors is in general neither linear nor continuous.
\begin{example}[Discontinuity] \label{exam:discontinuity}
Suppose that $\HH=\mathbb{R}^2$, set $U_{1}=\mathbb{R}\cdot (1,0)$, 
and set $U_{2}:=\mathbb{R}\cdot(1,1)$. Suppose that 
$\mathcal{S}=\{\Id, \R_{U_{1}}, \R_{U_{2}}\}$ or that $\mathcal{S}=\{\Id, \R_{U_{1}}, \R_{U_{2}}\R_{U_{1}}\}$. Let
$\overline{x} =(1,0)$ and let $(\forall k \in \mathbb{N})$
$x_{k}=(1,\frac{1}{k+1})$. As \cref{fig:CCSDiscontinuous.png} illustrates,
$\CC{\mathcal{S}}\overline{x} = \big(\frac{1}{2},\frac{1}{2}\big)$ and
$(\forall k \in \mathbb{N}) $ $\CC{\mathcal{S}}x_{k} = (0,0)$. Hence,
\begin{align*} 
  \lim_{k \to \infty} \CC{\mathcal{S}}x_{k} = (0,0) \neq
\big(\tfrac{1}{2},\tfrac{1}{2}\big)= \CC{\mathcal{S}}\overline{x}, 
\end{align*}
which implies that $\CC{\mathcal{S}}$ is not continuous at $\overline{x}$. 
By \cref{prop:FiniteIdresult}, the demiclosedness principle holds for 
$\CC{\mathcal{S}}$. 
\begin{figure}[H] 
\begin{center}\includegraphics[scale=0.18]{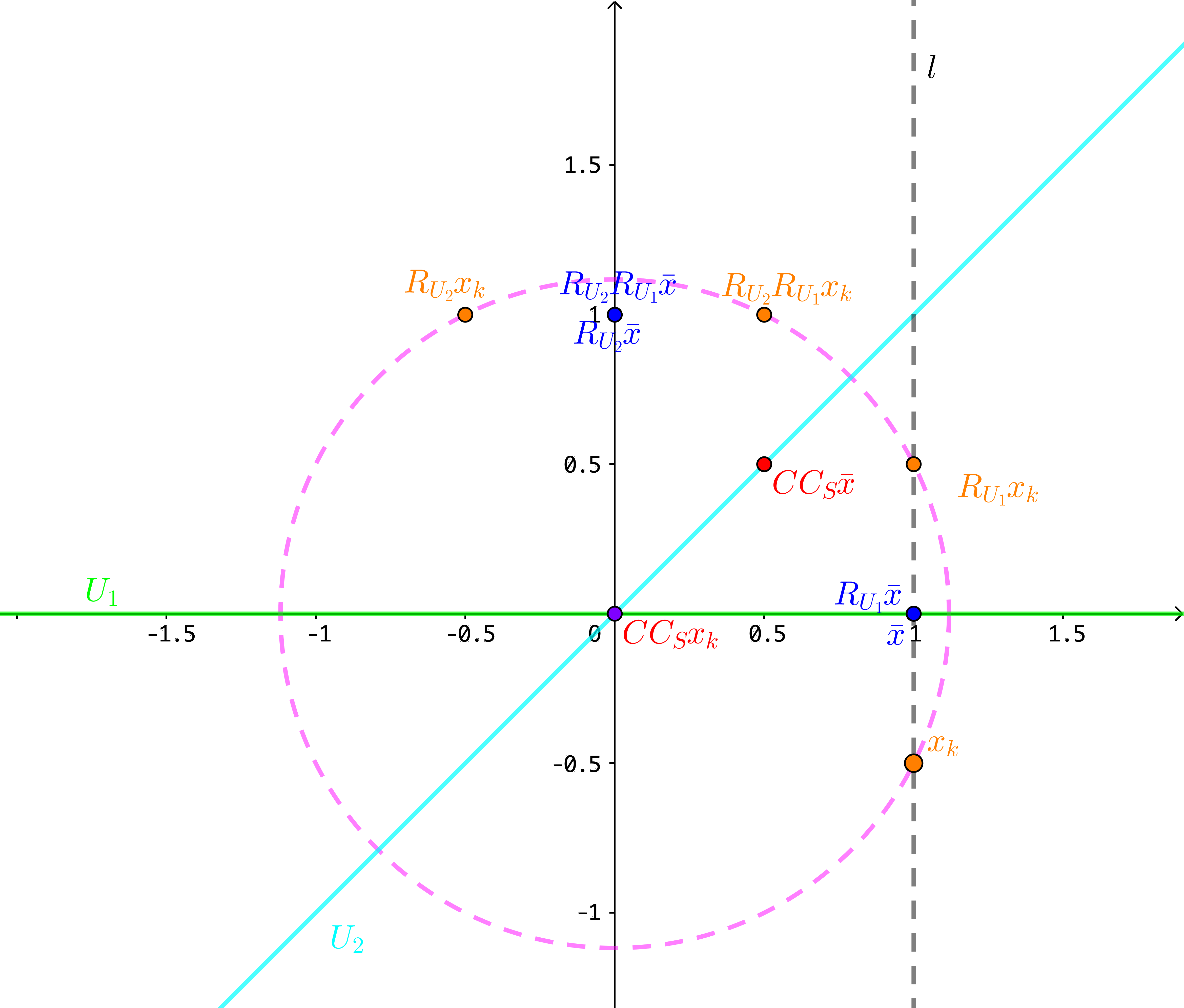}
\end{center}
\caption{\cref{exam:discontinuity} provides a discontinuous 
$\CC{\mathcal{S}}$ in $\mathbb{R}^{2}$.}  \label{fig:CCSDiscontinuous.png}
\end{figure}
\end{example}

\begin{example}[Nonlinearity]  \label{exam:not:linear}
Suppose that $\HH=\mathbb{R}^2$, 
set $U_{1}=\mathbb{R}\cdot (1,0)$ and set $U_{2}=\mathbb{R}\cdot(1,1)$. 
Suppose that $\mathcal{S}=\{\Id, \R_{U_{1}}, \R_{U_{2}}\}$ or that $\mathcal{S}=\{\Id,
\R_{U_{1}}, \R_{U_{2}}\R_{U_{1}}\}$. Let $x =(1,0)$ and $y=(1,-1)$. As \cref{fig:Not_linear.png} illustrates,
\begin{align*}
\CC{\mathcal{S}}x + \CC{\mathcal{S}}y =\big(\tfrac{1}{2},\tfrac{1}{2} \big) +
(0,0) \neq (0,0) = \CC{\mathcal{S}}(x+y),
\end{align*}
which shows that $\CC{\mathcal{S}}$ is not linear. 
By \cref{prop:FiniteIdresult}, the demiclosedness principle holds for 
$\CC{\mathcal{S}}$. 

\begin{figure}[H] 
\begin{center}\includegraphics[scale=0.11]{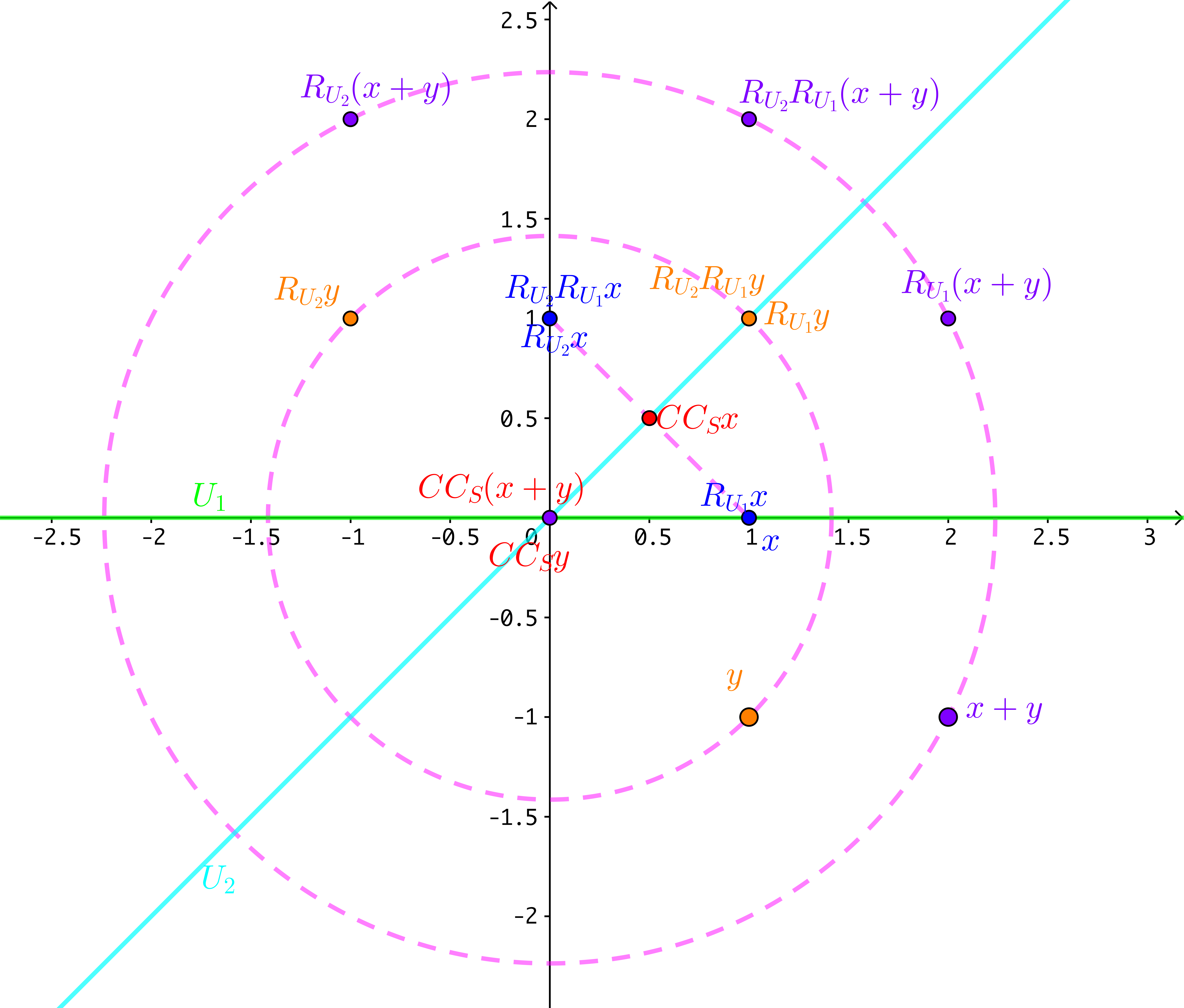}
\end{center}
\caption{\cref{exam:not:linear} presents a nonlinear
$\CC{\mathcal{S}}$ in $\mathbb{R}^{2}$} \label{fig:Not_linear.png} \label{Not linear.png}
\end{figure}
\end{example}

%---------------------------------------------Circumcenter mapping induced by projectors---------------------------------------------
\section{Circumcenter mappings induced by projectors} \label{sec:CircumMappingProjectots}

In this section, we uphold the notations that
\begin{empheq}[box=\mybluebox]{equation*}
\Omega = \Big\{ \R_{U_{i_{r}}}\cdots \R_{U_{i_{2}}}\R_{U_{i_{1}}}  ~\Big|~ r \in \mathbb{N}, ~\mbox{and}~ i_{1}, \ldots,  i_{r} \in \{1, \ldots,m\}    \Big\} \quad \text{and} \quad  \Id \in   \mathcal{S} \subseteq \Omega.
\end{empheq}
In addition, set 
\begin{empheq}[box=\mybluebox]{equation*}
\Theta = \Big\{ \Pro_{U_{i_{r}}}\cdots \Pro_{U_{i_{2}}}\Pro_{U_{i_{1}}}  ~\Big|~ r \in \mathbb{N}, ~\mbox{and}~ i_{1}, \ldots,  i_{r} \in \{1, \ldots,m\}    \Big\}.
\end{empheq}
By the empty product convention,  $\prod^{0}_{j=1}\Pro_{U_{i_{j}}} =\Id$. Hence $\Id \in \Theta $.
Specifically, we assume that
\begin{empheq}[box=\mybluebox]{equation*}
 \Id \in  \widehat{\mathcal{S}} \subseteq \aff \Theta.
\end{empheq}

\subsection{Proper circumcenter mappings induced by projectors}
First, we present some cases when $\CC{\widehat{\mathcal{S}}}$ is proper.

\begin{proposition} \label{prop:CWProLinComWellDefine}
Let $\alpha \in \mathbb{R}$. Assume that
\begin{align*}  
\widehat{\mathcal{S}}=\{\Id, (1-\alpha) \Id + \alpha \Pro_{U_{1}}, \ldots, (1-\alpha) \Id + \alpha \Pro_{U_{m}} \},
\end{align*}
and that
\begin{align*}  
\mathcal{S} = \{\Id, \R_{U_{1}}, \ldots,  \R_{U_{m}}\}.
\end{align*}
Then $\CC{\widehat{\mathcal{S}}}$ is proper. Moreover,
\begin{align} \label{eq:prop:CWProLinComWellDefine}
(\forall x \in \mathcal{H}) \quad \CC{\widehat{\mathcal{S}}}x =
\tfrac{\alpha}{2} \CC{\mathcal{S}}x + \big(1 - \tfrac{\alpha}{2}\big)x \in \mathcal{H}.
\end{align}
\end{proposition}

\begin{proof}
Apply \cref{prop:proper:affRU} with $\alpha$ replaced by $\frac{\alpha}{2}$.
\end{proof}

Taking $\alpha =1$ in \cref{prop:CWProLinComWellDefine}, 
we deduce the next result. 

\begin{corollary} \label{cor:CWProLinComWellDefine}
Assume that $\widehat{\mathcal{S}}=\{\Id, \Pro_{U_{1}}, \ldots, \Pro_{U_{m-1}},\Pro_{U_{m}} \}$. Then $\CC{\widehat{\mathcal{S}}}$ is proper, that is for every $x \in \mathcal{H}$, there exists unique $\CC{\widehat{\mathcal{S}}}x \in \mathcal{H}$ satisfying
\begin{enumerate}
\item $\CC{\widehat{\mathcal{S}}}(x) \in \aff\{x, \Pro_{U_{1}}(x), \ldots,  \Pro_{U_{m-1}}(x), \Pro_{U_{m}}(x)\}$ 
\item $\norm{\CC{\widehat{\mathcal{S}}}(x)-x}=\norm{\CC{\widehat{\mathcal{S}}}(x)-\Pro_{U_{1}}(x)}=\cdots =\norm{\CC{\widehat{\mathcal{S}}}(x)-\Pro_{U_{m-1}}(x)}=\norm{\CC{\widehat{\mathcal{S}}}(x)-\Pro_{U_{m}}(x)}$.
\end{enumerate}
\end{corollary}

\begin{proposition} \label{prop:PU1PU2PU1:Proper}
Assume that $U_{2}$ is linear 
and that $\widehat{\mathcal{S}}=\{\Id, \Pro_{U_{1}}, \Pro_{U_{2}}\Pro_{U_{1}}
\}$. Then $\CC{\widehat{\mathcal{S}}}$ is proper.
\end{proposition}
\begin{proof}
Let $x \in \mathcal{H}$. If $\card \big( \widehat{\mathcal{S}}(x) \big) \leq 2$, by \cref{prop:form:m2:Oper}, $\CC{\widehat{\mathcal{S}}}x \in \mathcal{H}$. Now assume $\card \big( \widehat{\mathcal{S}}(x) \big) =3$. If $x, \Pro_{U_{1}}x, \Pro_{U_{2}}\Pro_{U_{1}}x $ are affinely independent, by \cref{fact:clform:three}, $\CC{\widehat{\mathcal{S}}}x \in \mathcal{H}$. 

Assume that 
\begin{align} \label{eq:prop:PU1PU2PU1:Properaffde}
x, \Pro_{U_{1}}x, \Pro_{U_{2}}\Pro_{U_{1}}x~\text{are affinely dependent}.
\end{align}
 Note that $\card \big( \widehat{\mathcal{S}}(x) \big) =3$ implies that
 $\Pro_{U_{1}}x -x \neq 0$; moreover, \cref{eq:prop:PU1PU2PU1:Properaffde}
 yields that there exists $\alpha \neq 1$ such that
\begin{align} \label{eq:prop:PU1PU2PU1:Proper:LineDep}
 \Pro_{U_{2}}\Pro_{U_{1}}x -x = \alpha (\Pro_{U_{1}}x -x).
\end{align}
Because $U_{2}$ is linear subspace, $\Pro_{U_{2}}$ is linear. Applying
to both sides of \cref{eq:prop:PU1PU2PU1:Proper:LineDep} the projector 
$\Pro_{U_{2}}$, we obtain
\begin{align}
&~\Pro_{U_{2}} \Pro_{U_{2}}\Pro_{U_{1}}x - \Pro_{U_{2}}x = \alpha ( \Pro_{U_{2}}\Pro_{U_{1}}x - \Pro_{U_{2}}x) \nonumber \\
\Longrightarrow &~ \Pro_{U_{2}}\Pro_{U_{1}}x - \Pro_{U_{2}}x = \alpha ( \Pro_{U_{2}}\Pro_{U_{1}}x - \Pro_{U_{2}}x) \quad (\text{by \cref{lem:PU:RUIdempotent}\cref{lem:PUIdempotent}}) \nonumber \\
\Longrightarrow &~ (1-\alpha) \Pro_{U_{2}}\Pro_{U_{1}}x =   (1-\alpha) \Pro_{U_{2}}x \nonumber \\
\Longrightarrow &~ \Pro_{U_{2}}\Pro_{U_{1}}x = \Pro_{U_{2}}x. \quad ( \alpha \neq 1)  \label{eq:PU2PU1ALPHA}
\end{align}
Combining $\card \big( \widehat{\mathcal{S}}(x) \big) =3$ with
\cref{eq:prop:PU1PU2PU1:Properaffde} and \cref{eq:PU2PU1ALPHA}, 
we deduce that $x,
\Pro_{U_{1}}x, \Pro_{U_{2}}x$ are pairwise distinct and affinely dependent.
Applying \cref{cor:CWProLinComWellDefine} to $m=2$, we obtain $\CCO{(\{x,
\Pro_{U_{1}}x, \Pro_{U_{2}}x \})} \in \mathcal{H}$. 
But this contradicts \cref{fact:clform:three}.
Therefore, $\dom \CC{\widehat{\mathcal{S}}} = \mathcal{H}$. 
\end{proof}

\begin{proposition} \label{prop:PU1U2:Proper}
Assume that $U_{2}$ is linear and that $\widehat{\mathcal{S}}=\{\Id, \Pro_{U_{2}}, \Pro_{U_{2}}\Pro_{U_{1}} \}$. Then $\CC{\widehat{\mathcal{S}}}$ is proper.
\end{proposition}
\begin{proof}
Let $x \in \mathcal{H}$. Similarly to the proof in
\cref{prop:PU1PU2PU1:Proper}, we arrive at a contradiction for the
case where $\card \big( \widehat{\mathcal{S}}(x) \big) =3$ and there exists
$\alpha \neq 1$ such that
\begin{align} \label{eq:prop:PU1U2:Proper:LineDep}
\Pro_{U_{2}}\Pro_{U_{1}}x-x=\alpha ( \Pro_{U_{2}}x-x ).
\end{align}
As in the proof of \cref{prop:PU1PU2PU1:Proper}, we apply to 
both sides of  \cref{eq:prop:PU1U2:Proper:LineDep} the projector 
$\Pro_{U_{2}}$. Then
\begin{align*}
& \Pro_{U_{2}} \Pro_{U_{2}}\Pro_{U_{1}}x-\Pro_{U_{2}}x =\alpha (\Pro_{U_{2}} \Pro_{U_{2}}x- \Pro_{U_{2}} x)  \\
\Longrightarrow &~ \Pro_{U_{2}}\Pro_{U_{1}}x-\Pro_{U_{2}}x= \alpha (\Pro_{U_{2}}x-\Pro_{U_{2}}x) =0 \quad (\text{by \cref{lem:PU:RUIdempotent}\cref{lem:PUIdempotent})}\\
\Longrightarrow & \Pro_{U_{2}}\Pro_{U_{1}}x=\Pro_{U_{2}}x.
\end{align*}
which contradicts $\card \big( \widehat{\mathcal{S}}(x) \big) =3$. 
\end{proof}

\subsection{Improper circumcenter mappings induced by projectors} \label{Sec:Subsec:ImproProjec}

Propositions \ref{prop:CWProLinComWellDefine}, \ref{prop:PU1PU2PU1:Proper},
and \ref{prop:PU1U2:Proper} prompt the following question:

\begin{question} \label{ques:AffComb:Projec}
Suppose that $\{ \alpha_{1}, \alpha_{2} \} \subseteq \mathbb{R}
\smallsetminus \{0, 1\}$ and that at least one of $\alpha_{1}, \alpha_{2}$ is
not $2$.\footnotemark ~Assume that $\widehat{\mathcal{S}} = \Big\{ \Id, (1 -
\alpha_{1}) \Id + \alpha_{1} \Pro_{U_{1}}, \big ( (1 -\alpha_{2}) \Id +
\alpha_{2}\Pro_{U_{2}} \big) \circ \big( (1 - \alpha_{1}) \Id + \alpha_{1}
\Pro_{U_{1}} \big)\Big\}$ or $\widehat{\mathcal{S}} = \Big\{ \Id, (1 -
\alpha_{1}) \Id + \alpha_{1} \Pro_{U_{2}}, \big ( (1 -\alpha_{2}) \Id +
\alpha_{2}\Pro_{U_{2}} \big) \circ \big( (1 - \alpha_{1}) \Id + \alpha_{1}
\Pro_{U_{1}} \big)\Big\}$. Is $\CC{\widehat{\mathcal{S}}}$ proper?
\end{question}
\footnotetext{For $i \in \{1,2\}$, when $\alpha_{i}$ is $0, 1$, or $2$, then $(1  - \alpha_{i})  \Id + \alpha_{i} \Pro_{U_{1}}$ is $\Id, \Pro_{U_{1}}$, or $\R_{U_{1}}$ respectively. In these special cases, the answer for \cref{ques:AffComb:Projec} is positive (see \cref{prop:CWProLinComWellDefine} and \cref{thm:CCS:proper}\cref{thm:CCS:proper:prop}). }
The following example demonstrates 
that the answer to \cref{ques:AffComb:Projec} is negative. 

\begin{example} \label{exam:UUPUIMPROP}
Assume that $m=2$ and that $U:=U_{1} = U_{2} \subsetneqq \mathcal{H}$ and
$\{\alpha_{1}, \alpha_{2} \} \subseteq \mathbb{R}$. Assume further that
$\widehat{\mathcal{S}} = \big\{ \Id, (1 - \alpha_{1}) \Id + \alpha_{1}
\Pro_{U}, \big ( (1 -\alpha_{2}) \Id + \alpha_{2}\Pro_{U} \big) \circ \big(
(1 - \alpha_{1}) \Id + \alpha_{1} \Pro_{U} \big)\big\}$.
Then $\CC{\widehat{\mathcal{S}}}$ is improper if and only if  $ \alpha_{1} \neq 0,  \alpha_{1} \neq 1, \alpha_{2} \neq 0$ and $\alpha_{2} \neq \frac{\alpha_{1}}{\alpha_{1} -1}$.
\end{example}
\begin{proof}
Since $\R_{U} =2 \Pro_{U} -\Id $, we deduce that
\begin{align*}
\widehat{\mathcal{S}} & = \Big\{ \Id, (1  - \alpha_{1})  \Id + \alpha_{1} \Pro_{U},  \big ( (1 -\alpha_{2}) \Id  + \alpha_{2}\Pro_{U} \big) \circ  \big( (1  - \alpha_{1})  \Id + \alpha_{1} \Pro_{U} \big) \Big\} \\
& = \Big\{ \Id,  \Id + \alpha_{1} ( \Pro_{U} - \Id ),  \big( \Id  + \alpha_{2} ( \Pro_{U} -\Id ) \big) \circ  \big(  \Id + \alpha_{1} ( \Pro_{U}- \Id ) \big) \Big\} \\
& =   \Big\{ \Id,  \Id + \frac{\alpha_{1}}{2} ( \R_{U} - \Id ),  \big( \Id  + \frac{\alpha_{2}}{2} ( \R_{U} -\Id ) \big) \circ  \big(  \Id + \frac{\alpha_{1}}{2}  ( \R_{U}- \Id ) \big) \Big\}. 
\end{align*}
The result now follows from the assumptions above and \cref{exam::IMp:affRURU}. 
\end{proof}

Next, we present further improper instances of 
$\CC{\widehat{\mathcal{S}}}$, where $\Id \in  \widehat{\mathcal{S}} \subseteq \aff \Theta$. 

\begin{example} \label{exam:ImproperColinear}
Assume that $\mathcal{H}=\mathbb{R}^{2}$, that $m=2$, 
that $U_{1} = \mathbb{R} \cdot (1,0)$, and that $U_{2}= \mathbb{R} \cdot (1,2)$.
Assume further that $\widehat{\mathcal{S}}=\{\Id, \Pro_{U_{2}}\Pro_{U_{1}},
\Pro_{U_{2}}\Pro_{U_{1}}\Pro_{U_{2}}\Pro_{U_{1}}\}$. Take $x=(2, 4) \in
U_{2}$. As \cref{fig:Pro_ContExam_Nonexist3Point.png} illustrates, $x$,
$\Pro_{U_{2}}\Pro_{U_{1}}x$, and
$\Pro_{U_{2}}\Pro_{U_{1}}\Pro_{U_{2}}\Pro_{U_{1}}x$ are pairwise distinct and
colinear. By \cref{thm:Proper3}, $\CC{\widehat{\mathcal{S}}}$ is improper.
\begin{figure}[H] 
\begin{center}\includegraphics[scale=1.1]{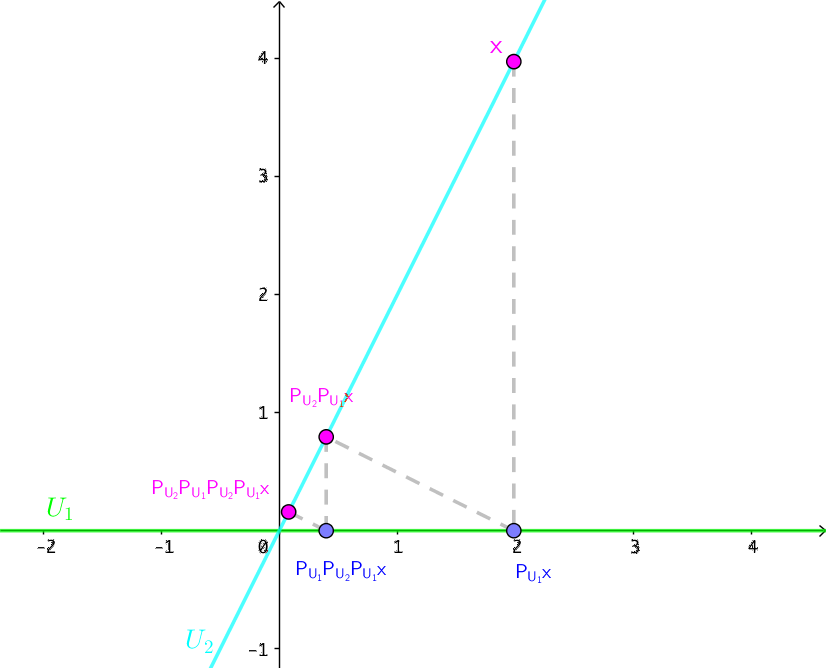}
\end{center}
\caption{\cref{exam:ImproperColinear} illustrates $\CC{\widehat{\mathcal{S}}}x = \varnothing$ for the colinear case.} \label{fig:Pro_ContExam_Nonexist3Point.png}
\end{figure}
\end{example}

\begin{example} \label{exam:Improper:Noncolinear}
Assume that $\mathcal{H}=\mathbb{R}^{2}$, that $m=2$,
that $U_{1} = \mathbb{R} \cdot (1,0)$, and that $U_{2} = \mathbb{R} \cdot (1,1)$.
Assume further that $\widehat{\mathcal{S}}=\{\Id, \Pro_{U_{1}}, \Pro_{U_{2}},
\Pro_{U_{2}}\Pro_{U_{1}}\}$. Take $x=(4,2)$ and set $\mathcal{K}=\{\Id,
\Pro_{U_{1}}, \Pro_{U_{2}}\}$. Clearly, $\Pro_{U_{2}}\Pro_{U_{1}}x-x \in
\mathbb{R}^{2}= \spn \{ \Pro_{U_{1}}x-x, \Pro_{U_{2}}x-x \}$, which implies
that $\aff (\mathcal{K}(x)) = \aff (\widehat{\mathcal{S}}(x))$. By
\cref{fact:unique:BasisPformula}, if $\CC{\widehat{\mathcal{S}}}x \in
\mathcal{H}$, then $\CC{\widehat{\mathcal{S}}}x = \CC{\mathcal{K}}x$.
As \cref{fig:Pro_ContExam_Nonexist.png} shows, 
\begin{align*}
\norm{\CC{\mathcal{K}}x-x}=\norm{\CC{\mathcal{K}}x-\Pro_{U_{1}}x}=\norm{\CC{\mathcal{K}}x-\Pro_{U_{2}}x}\neq \norm{\CC{\mathcal{K}}x-\Pro_{U_{2}}\Pro_{U_{1}}x}.
\end{align*}
Hence $\CC{\widehat{\mathcal{S}}}x = \varnothing$, which implies that $\CC{\widehat{\mathcal{S}}}$ is improper.
\begin{figure}[H]
\begin{center}\includegraphics[scale=0.8]{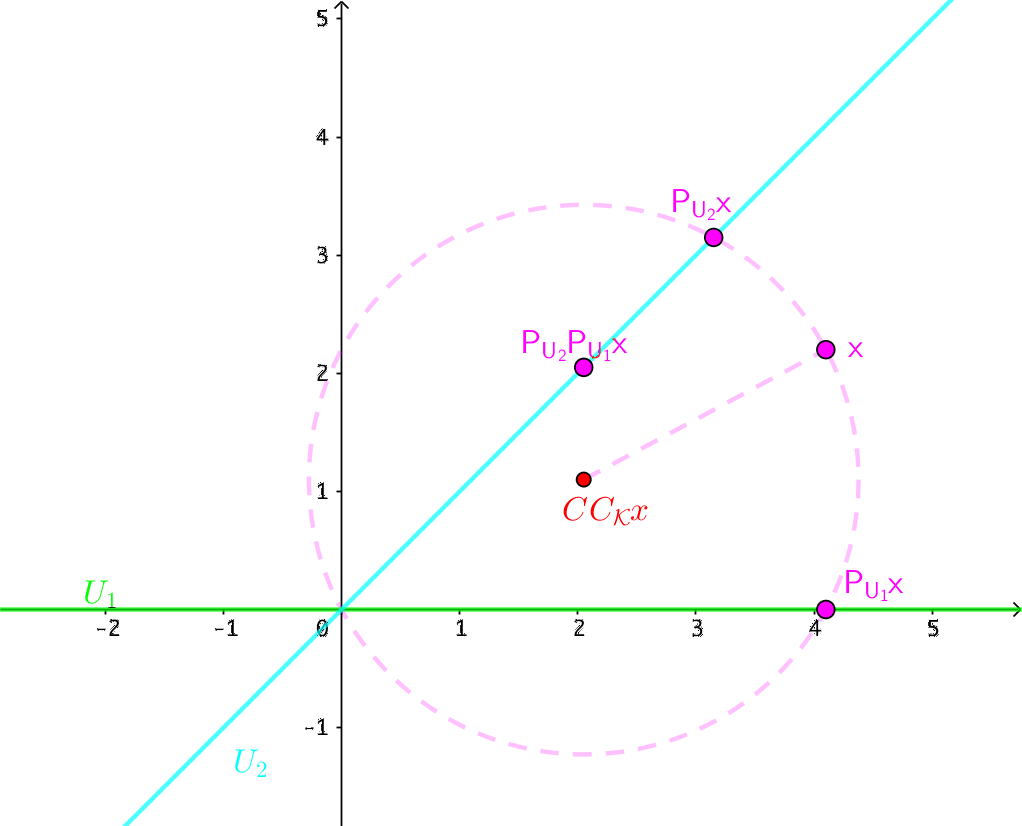}
\end{center}
\caption{\cref{exam:Improper:Noncolinear} illustrates
$\CC{\widehat{\mathcal{S}}}x
= \varnothing$ for the non-colinear case.}
\label{fig:Pro_ContExam_Nonexist.png}
\end{figure}
\end{example}

%---------------------------------------------More improper circumcenter mappings---------------------------------------------
\section{More improper circumcenter mappings induced by reflectors} \label{sec:MoreImproper}
In \cref{thm:CCS:proper}\cref{thm:CCS:proper:prop}, to prove
$\CC{\mathcal{S}}$ is proper, we required that
\begin{align} \label{eq:Cond:CCS:Prop}
U_{1}, \ldots, U_{m}~\text{are closed affine subspaces in}~\mathcal{H}~\text{with}~\cap^{m}_{i=1} U_{i} \neq \varnothing,
\end{align}
and that 
\begin{align} \label{eq:Cond:CCS:Prop:SIDOME}
\Id \in  \mathcal{S} \subseteq \Omega .
\end{align}
In \cref{subsec:ImpcircuMappingRefle}, we have already seen that when the
condition $\mathcal{S} \subseteq \Omega$ fails, the circumcenter mapping
induced by reflectors $\CC{\mathcal{S}}$ may be improper. In the remaining
part of this section, we consider two circumcenter mappings induced by
reflectors, where $m=2$ and $\mathcal{S}= \{\Id, \R_{U_{1}}, \R_{U_{2}}\}$ or
$\mathcal{S}= \{\Id, \R_{U_{1}}, \R_{U_{2}}\R_{U_{1}}\}$. We construct 
additional improper circumcenter mappings with the conditions in
\cref{eq:Cond:CCS:Prop} not being satisfied, which means that the conditions
\cref{eq:Cond:CCS:Prop} and \cref{eq:Cond:CCS:Prop:SIDOME} are sharp.

\subsection{Inconsistent cases} \label{sec:FailPropInconsis}
In this subsection, we focus on the case when 
$\cap^{m}_{i=1} U_{i} = \varnothing$. 
Let $U$ and $V$ be two nonempty, closed, convex (possibly nonintersecting) 
subsets of $\mathcal{H}$. A \emph{best approximation pair} relative to $(U,V)$ is
\begin{empheq}{equation*} 
\quad (a,b) \in U\times V \quad \text{such that} \quad  \norm{a-b}=\inf~ \norm{U-V}.
\end{empheq}
In \cite{BCL2004}, the authors used the Douglas--Rachford splitting operator
$T = \frac{\R_{V}\R_{U} +\Id}{2}$ to find a best approximation pair
relative to $(U,V)$.

\begin{fact} {\rm \cite[Theorem~3.13 and Remark~3.14(ii)]{BCL2004}} \label{fac:BesApproxPair}
Let $U$ be a closed affine subspace and let $V$ be a nonempty, closed, convex
set in $\mathcal{H}$ ($U, V$ are possibly non-intersecting). Suppose that best
approximation pairs relative to $(U,V)$ exist. Set $T:= \frac{\R_{V}\R_{U}
+\Id}{2}$. Let $x_{0} \in \mathcal{H}$ and set $x_{n} =T^{n}x_{0}$, for all
$n \in \mathbb{N}$. Then
\begin{align*}
\big( (\Pro_{V}\R_{U}x_{n},\Pro_{U}x_{n}) \big)_{n \in \mathbb{N}} \quad  \text{and} \quad \big( (\Pro_{V}\Pro_{U}x_{n},\Pro_{U}x_{n}) \big)_{n \in \mathbb{N}}
\end{align*} 
both converge weakly to best approximation pairs relative to $(U,V)$.
\end{fact}

The following examples show that even if both of $U_{1}, U_{2}$ are closed
affine subspaces, when $U_{1} \cap U_{2} = \varnothing$, the
operator $\CC{\mathcal{S}}$ may not be proper where $\mathcal{S}=\{\Id, \R_{U_{1}},
\R_{U_{2}}\}$ or $\mathcal{S}=\{\Id, \R_{U_{1}}, \R_{U_{2}}\R_{U_{1}}\}$.
(Notice that in \cref{exam:PointLine}, $U_{1}$ is even a compact set.)
Hence, we can not directly generalize \cref{fac:BesApproxPair} by the circumcenter mapping induced by reflectors. 

The results of the following examples in this section are easily from \cref{cor:T1T2T3Dom} and the proofs are omitted.

\begin{example}  \label{exam:PointLine}
Assume that $\mathcal{H} =\mathbb{R}^{2}$, 
that $U_{1}=\{(2,0)\}$, and that $U_{2} = \mathbb{R} \cdot (0,1)$. 
Set $\mathcal{S}_{1}=\{\Id, \R_{U_{1}}, \R_{U_{2}}\}$ and
$\mathcal{S}_{2}=\{\Id, \R_{U_{1}}, \R_{U_{2}}\R_{U_{1}}\}$. Then 
\begin{align*}
& \dom \CC{\mathcal{S}_{1}} =\big( \mathbb{R}^{2} \smallsetminus \mathbb{R}\cdot(1,0) \big ) \cup \big\{(2,0), (0,0) \big\}, \\
& \dom \CC{\mathcal{S}_{2}} =\big( \mathbb{R}^{2} \smallsetminus \mathbb{R}\cdot(1,0) \big ) \cup \big\{(2,0), (4,0)\big\}.
\end{align*}
\end{example}

\subsection{Non-affine cases} \label{sec:FailPropNon-affin}

One of the charming aspects of the Douglas--Rachford method is that it can be
used for general convex sets. In this subsection, we assume that
\begin{align*}
\mathcal{S}_{1}=\{\Id, \R_{U_{1}}, \R_{U_{2}}\} \quad \text{or} \quad  \mathcal{S}_{2}=\{\Id, \R_{U_{1}}, \R_{U_{2}}\R_{U_{1}}\}.
\end{align*}
We shall present examples in which the operator 
$\CC{\mathcal{S}}$ is improper,
with at least one of $U_{1}$ and $U_{2}$ not being an affine subspace while $U_{1}
\cap U_{2} \neq \varnothing$.

\begin{example} \label{exam:ConvConeLin}
Assume that $\mathcal{H} = \mathbb{R}^{2}$, that $U_{1}=
\mathbb{R}^{2}_{+}$, and that $U_{2}= (2,0)+\mathbb{R} \cdot (0,1)$. Then
\begin{align*}
& \dom \CC{\mathcal{S}_{1}} =  \mathbb{R}^{2} \smallsetminus \{(x,y) ~|~x<0 ~\text{and}~y \geq 0\}, \\
& \dom  \CC{\mathcal{S}_{2}} = \big( \mathbb{R}^{2} \smallsetminus \{(x,y) ~|~ x<0 ~\text{and}~y \geq 0 \} \big) \cup \{(-2,y) ~|~y \geq0 \} . 
\end{align*}
\end{example}

In the remainder of this subsection, we revisit the examples used in
\cite{BCS2017} to show the potential of the Circumcentering Douglas--Rachford
method, which are the iterations of the operator $\CC{\mathcal{S}_{2}}$.

\begin{example} \label{exam:BallLine}
Assume that $\mathcal{H} = \mathbb{R}^{2}$, that $U_{1}=\mathbf{B}[(0,0);1]$, and that $U_{2}=(1,0)+\mathbb{R} \cdot
(0,1)$. Then 
\begin{align*}
& \dom \CC{\mathcal{S}_{1}} =  \mathbb{R}^{2} \smallsetminus \{(x,0) ~|~ x <-1 \},\\
& \dom  \CC{\mathcal{S}_{2}} = \mathbb{R}^{2} \smallsetminus \{(x,0) ~|~ x<-3 ~\text{or}~  -3<x <-1\} . 
\end{align*}
\end{example}

\begin{example} \label{exam:Ball_LineMore}
Assume that $\mathcal{H} = \mathbb{R}^{2}$, 
that $U_{1}=\mathbf{B}[(0,0);1]$, 
and that $U_{2}=\mathbb{R} \cdot (0,1)$. Then
\begin{align*}
& \dom \CC{\mathcal{S}_{1}} =  \mathbb{R}^{2} \smallsetminus \{(x,0) ~|~ x <-1 ~\text{or}~x >1 \} ,\\
& \dom  \CC{\mathcal{S}_{2}} = \mathbb{R}^{2} \smallsetminus \{(x,0) ~|~ x<-2 ~\text{or}~  -2<x <-1 ~\text{or}~   1 <x < 2 ~\text{or}~ x >2\} . 
\end{align*}
\end{example}

\begin{example} \label{exam:BallBall:dom}
Assume that $\mathcal{H}= \mathbb{R}^{2}$, that $U_{1}=\mathbf{B}[(-1,0);1]$, and that $U_{2}=\mathbf{B}[(1,0);1]$.
Then  
\begin{equation*}
\dom \CC{\mathcal{S}_{1}} =  \mathbb{R}^{2} \smallsetminus \{(x,0) ~|~ x <-2 ~\text{or}~x >2 \},
\end{equation*}
\begin{equation*}
\{(x,0) ~|~ -6 \leq x \leq -4~\text{or}~ x \geq -2\} \subseteq  \dom \CC{\mathcal{S}_{2}},
\end{equation*}
\begin{equation*}
\{(x,0) ~|~x<-6 ~\text{or}~ -4 <x<-2\} \subseteq  \mathbb{R}^2\smallsetminus \big( \dom \CC{\mathcal{S}_{2}} \big). 
\end{equation*}
\end{example}

\begin{example} \label{exam:Ball_Ball:dom}
Assume that $\mathcal{H} = \mathbb{R}^{2}$, that $U_{1}=\mathbf{B}[(-1,0);2]$, and that $U_{2}=\mathbf{B}[(1,0);2]$. Then 
 \begin{equation*}
   \dom \CC{\mathcal{S}_{1}} =  \mathbb{R}^{2} \smallsetminus \{(x,0) ~|~ x <-3 ~\text{or}~x >3 \},
 \end{equation*}
 \begin{equation*}
   \{(x,0) ~|~ -9 \leq x \leq -5~\text{or}~ -3 \leq x \leq 3\} \subseteq  \dom \CC{\mathcal{S}_{2}},
\end{equation*}
 \begin{equation*}
  \{(x,0) ~|~x<-9 ~\text{or}~ -5 <x<-3 ~\text{or}~x >3\} \subseteq  \mathbb{R}^2\smallsetminus \big( \dom \CC{\mathcal{S}_{2}} \big).
 \end{equation*}
\end{example}

Finally, consider $U_{1}=\{(x,y) \in \mathbb{R}^{2} ~|~ (x+1)^{2}+y^{2} =
4\}$ and $U_{2}=\{(x,y) \in \mathbb{R}^{2} ~|~ (x-1)^{2}+y^{2} = 4\}$. Note that
neither $U _{1}$ nor $U_{2}$ is convex. For $\mathcal{S}=\{\Id,
\R_{U_{1}}, \R_{U_{2}}\R_{U_{1}}\}$ or $\mathcal{S}=\{\Id, \R_{U_{1}},
\R_{U_{2}}\}$, one can show that $
\dom \CC{\mathcal{S}} \subsetneqq \mathbb{R}^2$. 

\subsection{Impossibility to extend to maximally monotone operators} \label{sec:FailGenMaxiMono}
Assume that $\mathcal{S}= \{\Id, \R_{U_{1}}, \R_{U_{2}} \}$ or $\mathcal{S}=
\{\Id, \R_{U_{1}}, \R_{U_{2}}\R_{U_{1}}\}$. In order to show a counterexample
where the definition of $\CC{\mathcal{S}}$ fails to be directly generalized
to maximally monotone theory, we need the definition and facts below.

\begin{definition} {\rm \cite[Definition~23.1]{BC2017}}
Let $A : \mathcal{H} \rightarrow 2^{\mathcal{H}}$. The \emph{resolvent} of $A$ is
\begin{align*}
J_{A} = (\Id + A)^{-1}.
\end{align*}
\end{definition}

\begin{fact} {\rm \cite[Corollary~23.11]{BC2017}}
Let $A : \mathcal{H} \rightarrow 2^{\mathcal{H}}$ be maximally monotone and let $\gamma \in \mathbb{R}_{++}$. Then the following hold.
\begin{enumerate}
\item $J_{\gamma A} : \mathcal{H} \rightarrow \mathcal{H}$ and $\Id - J_{\gamma A} : \mathcal{H} \rightarrow \mathcal{H}$ are firmly nonexpansive and maximally monotone.
\item The \emph{reflected resolvent} 
\begin{align*}
R_{\gamma A}: \mathcal{H} \rightarrow \mathcal{H} : x \mapsto 2 J_{\gamma A} x -x.
\end{align*}
is nonexpansive.
\end{enumerate}
\end{fact}

\begin{fact} {\rm \cite[Corollary~20.28]{BC2017}} \label{fact:ContMaxim}
Let $A : \mathcal{H} \rightarrow \mathcal{H}$ be monotone and continuous. Then $A$ is maximally monotone.
\end{fact}

\begin{fact} {\rm \cite[Corollary~20.26]{BC2017}}
Let $C$ be a nonempty closed convex subset of $\mathcal{H}$. Then $N_{C}$ is maximally monotone.
\end{fact}

\begin{fact} {\rm \cite[Corollary~23.4]{BC2017}} \label{fac:JNCPC}
Let $C$ be a nonempty closed convex subset of $\mathcal{H}$.  Then
\begin{align*}
J_{N_{C}}=(\Id +N_{C})^{-1}= \Pro_{C}.
\end{align*}
\end{fact}

By \cref{fac:JNCPC}, $\R_{U_{1}}= 2 \Pro_{U_{1}} -\Id = 2 J_{N_{U_{1}}} -\Id$ and $\R_{U_{2}}= 2 \Pro_{U_{2}} -\Id = 2 J_{N_{U_{2}}} -\Id$. In these special cases, the reflectors are consistent with the corresponding reflected resolvent.

In the following examples, we replace the two maximally monotone operators
$N_{U_{1}}, N_{U_{2}}$ in the set $\mathcal{S} = \{\Id, 2 J_{N_{U_{1}}} -\Id,
2 J_{N_{U_{2}}} -\Id\}$ or $\mathcal{S} = \{\Id, 2 J_{N_{U_{1}}} -\Id, (2
J_{N_{U_{2}}} -\Id ) \circ (2 J_{N_{U_{1}}} -\Id)\}$ by $\alpha \Id$ and
$\beta \Id$ respectively, with $\alpha \geq 0$ and $\beta \geq 0$. By
\cref{fact:ContMaxim}, since $\alpha \geq 0$ and $\beta \geq 0$, we obtain
that $\alpha \Id$ and $\beta \Id$ are maximally monotone operators. We shall
characterize the improperness of the new operator $\CC{\mathcal{S}}$.

\begin{example} \label{exam:MaxiMonoFail}
Assume that $\{0\}\subsetneqq \mathcal{H}$. 
Set $A=\alpha \Id $ and $B=\beta \Id $, 
where $\alpha \geq 0$ and $\beta \geq 0$. Further set 
\begin{align*}
\mathcal{S}_{1}=\{\Id, R_{A}, R_{B}\} \quad \text{and} \quad \mathcal{S}_{2} = \{\Id, R_{A}, R_{B}R_{A}\}.
\end{align*}
Then $\CC{\mathcal{S}_{1}}$ is improper if and only if $\alpha \neq 0$, $\beta \neq 0$ and $\alpha \neq \beta$. Moreover, $\CC{\mathcal{S}_{2}}$ is improper if and only if $\alpha \neq 0$, $\alpha \neq 1$, $\beta \neq 0$ and $\alpha \neq - \beta$. 
\end{example}

\begin{proof}
The definitions yield
\begin{align*}
&J_{A} =(A+\Id)^{-1}=\big((\alpha +1) \Id\big)^{-1}=\frac{1}{\alpha +1} \Id; \quad R_{A}=2 J_{A}- \Id= \frac{2}{\alpha +1} \Id -\Id= \frac{1-\alpha}{\alpha +1} \Id; \\
&J_{B} =(B+\Id)^{-1}=(\big((\beta +1) \Id\big)^{-1}=\frac{1}{\beta +1} \Id; \quad R_{B}=2 J_{B}- \Id= \frac{2}{\beta +1} \Id -\Id= \frac{1-\beta}{\beta +1} \Id.
\end{align*}

Let $x \in \mathcal{H} \smallsetminus{0}$. Now
\begin{subequations}  \label{eq:xrarbrba}
\begin{align}
& x=R_{A}x \Longleftrightarrow x = \frac{1-\alpha}{\alpha +1}x \Longleftrightarrow 1=  \frac{1-\alpha}{\alpha +1} \Longleftrightarrow \alpha =0;\\
& x= R_{B}x \Longleftrightarrow x= \frac{1-\beta}{\beta +1}x \Longleftrightarrow \beta =0;\\
& R_{A}x =R_{B}x \Longleftrightarrow \frac{1-\alpha}{\alpha +1}x = \frac{1-\beta}{\beta +1}x \Longleftrightarrow  (1-\alpha)(\beta+1)=(\alpha+1)(1-\beta)\Longleftrightarrow \alpha =\beta ;\\
& x =R_{B}R_{A}x  \Longleftrightarrow x= \frac{1-\alpha}{\alpha +1} \frac{1-\beta}{\beta +1}x \Longleftrightarrow (\alpha+1)(\beta+1)=  (1-\alpha) (1-\beta) \Longleftrightarrow \alpha = -\beta;\\
& R_{A}x =R_{B}R_{A}x \Longleftrightarrow   \frac{1-\alpha}{\alpha +1}x = \frac{1-\alpha}{\alpha +1} \frac{1-\beta}{\beta +1}x \Longleftrightarrow  \alpha =1 ~\text{or}~ 1=  \frac{1-\beta}{\beta +1} \Longleftrightarrow \alpha =1 ~\text{or}~ \beta =0.
\end{align}
\end{subequations}

\enquote{$\Longrightarrow$}: According to the previous analysis, in both of the assertions, the contrapositive of the required results  follow from \cref{prop:form:m2:Oper}. 

\enquote{$\Longleftarrow$}: Assume $\alpha \neq 0$, $\beta \neq 0$ and $\alpha \neq \beta$. Then 
\begin{align*}
\aff (\mathcal{S}_{1}(x)) & = \aff \{x, R_{A}x, R_{B}x\}= x +\spn \{R_{A}x-x, R_{B}x-x\} \\
& =x + \spn\Big\{\frac{-2\alpha}{\alpha +1}x,\frac{-2\beta}{\beta+1}x \Big\} \\
& = \mathbb{R}\cdot x.
\end{align*}

Let $x \in \mathcal{H} \smallsetminus \{0\}$. We observe that
\begin{subequations} \label{eq:exam:MaxiMonoFail:yS1}
\begin{align}
& \Big(\exists y \in \aff (\mathcal{S}_{1}(x))\Big) \quad  \norm{y- x} =\norm{y-R_{A}x}= \norm{y-R_{B}x} \\
\Longleftrightarrow & (\exists t \in \mathbb{R}) \quad \norm{t x- x} = \norm{t x- \frac{1-\alpha}{\alpha +1}x} =\norm{t x- \frac{1-\beta}{\beta +1}x} \\
\Longleftrightarrow & (\exists t \in \mathbb{R}) \quad  |t -1|= \Big|t -  \frac{1-\alpha}{\alpha +1} \Big|= \Big| t - \frac{1-\beta}{\beta +1} \Big|. \quad (\text{by}~x \neq 0)
\end{align}
\end{subequations}
On the other hand, combining the assumptions with \cref{cor:mathbbRcard3} and \cref{eq:xrarbrba}, we obtain that
\begin{align} \label{eq:exam:MaxiMonoFail:not}
(\cancel{\exists} t \in \mathbb{R}) \quad |t -1|= \Big|t -  \frac{1-\alpha}{\alpha +1} \Big|= \Big| t - \frac{1-\beta}{\beta +1} \Big|.
\end{align}
Hence, 
\begin{align*}
(\forall x \in \mathcal{H} \smallsetminus \{0\})  \quad \CC{\mathcal{S}_{1}}x = \varnothing.
\end{align*}

Assume $\alpha \neq 0$, $\alpha \neq 1$, $\beta \neq 0$ and $\alpha \neq - \beta$. A similar
proof shows that for every $x \in \mathcal{H} \smallsetminus \{0\}$, there is no point $y \in \aff (\mathcal{S}_{2}(x))$, such
that
$\norm{y- x} =\norm{y-R_{A}x}= \norm{y-R_{B}R_{A}x}$, 
which implies that 
$(\forall x \in \mathcal{H} \smallsetminus \{0\})$ 
$\CC{\mathcal{S}_{2}}x = \varnothing$.
\end{proof}

Arguing similarly to the proof of the previous result,
we also obtain the following result:

\begin{example} \label{exam:MaxiMonoConstFail}
Assume that $\{0\} \subsetneqq \mathcal{H}$. 
Let $\{a,b\} \subseteq \mathbb{R}$. 
Set $A \equiv a $, i.e., $(\forall x \in \mathcal{H})$ $Ax =a$, and $B \equiv
b $. Furthermore, set 
\begin{align*}
\mathcal{S}_{1}=\{\Id, R_{A}, R_{B}\} \quad \text{and} \quad \mathcal{S}_{2} = \{\Id, R_{A}, R_{B}R_{A}\}.
\end{align*}
Then $\CC{\mathcal{S}_{1}}$ is improper if and only if $a \neq 0$, $b \neq 0$ and $a \neq b$. Moreover, $\CC{\mathcal{S}_{2}}$ is improper if and only if $a \neq 0$, $b \neq 0$ and $a \neq - b$. 
\end{example}

The example above shows that there is no direct way to generalize the definition of $\CC{\mathcal{S}}$ to 
maximally monotone theory.

%---------------------------------------------Acknoledgments---------------------------------------------		
\section*{Acknowledgments}
HHB and XW were partially supported by NSERC Discovery Grants.

%---------------------------------------------References---------------------------------------------				
\addcontentsline{toc}{section}{References}

\bibliographystyle{abbrv}

\begin{thebibliography}{19}

\bibitem{BCNPW2014}
{\sc H.~H. Bauschke, J.~Y. Bello~Cruz, T.~T.~A. Nghia, H.~M. Phan, and
  X.~Wang}, {\em The rate of linear convergence of the {D}ouglas-{R}achford
  algorithm for subspaces is the cosine of the {F}riedrichs angle}, 
  J.~Approx.\ Theory~185 (2014), pp.~63--79.

\bibitem{BC2017}
{\sc H.~H. Bauschke and P.~L. Combettes}, {\em Convex Analysis and Monotone
 Operator Theory in {H}ilbert Spaces}, CMS Books in Mathematics, Springer, second~ed., 2017.

\bibitem{BCL2004}
{\sc H.~H. Bauschke, P.~L. Combettes, and D.~R. Luke}, {\em Finding best
  approximation pairs relative to two closed convex sets in {H}ilbert spaces},
  J.~Approx.\ Theory~127 (2004), pp.~178--192.
  

\bibitem{BOyW2018}
{\sc H.~H. Bauschke, H.~Ouyang, and X.~Wang}, {\em On circumcenters of finite
  sets in {H}ilbert spaces}, Linear Nonlinear Anal.~4 (2018), to appear.

\bibitem{BCW2015}
{\sc H.~H. Bauschke,   C.~Wang,   X.~Wang, and
J.~Xu}, {\em On subgradient projectors}, 
SIAM J.~Optim.~25 (2015), pp.~1064--1082.

\bibitem{BCS2017}
{\sc R.~Behling, J.~Y. Bello~Cruz, and L.-R. Santos}, {\em Circumcentering the
  {D}ouglas--{R}achford method}, Numer.\ Algorithms~78 (2018), pp.~759--776.

\bibitem{BCS2018}
{\sc R.~Behling, J.~Y. Bello~Cruz, and L.-R. Santos},  {\em On the linear
  convergence of the circumcentered-reflection method}, 
  Oper.\ Res.\ Lett.~46 (2018), pp.~159--162.

\bibitem{B1968}
{\sc F.~E. Browder},  {\em Semicontractive and semiaccretive nonlinear mappings in Banach spaces}, 
Bulletin of the AMS~74 (1968), pp.~660--665.
  


\bibitem{D2012}
{\sc F.~Deutsch}, {\em Best Approximation in Inner Product Spaces}, CMS Books in Mathematics,
  Springer-Verlag, New York, 2012.
  
  
\bibitem{R1970}
{\sc R.~T. Rockafellar}, {\em Convex Analysis}, Princeton Landmarks in
  Mathematics, Princeton University Press, Princeton, NJ, 1997.
\newblock Reprint of the 1970 original, Princeton Paperbacks.

		
\end{thebibliography}

\end{document}